\documentclass[sts,preprint,reqno]{imsart}

\RequirePackage[OT1]{fontenc}
\RequirePackage{amsthm,amsmath}
\RequirePackage[colorlinks,citecolor=blue,urlcolor=blue]{hyperref}
\usepackage[english]{babel}
\usepackage[normalem]{ulem}
\usepackage{graphicx}   
\usepackage{amsfonts}
\usepackage{amssymb}    
\usepackage{float}
\usepackage{color}

\def\X{\mathbf{X}}
\def\J{\mathcal{J}}
\def\K{\mathcal{K}}
\def\hK{\widehat{\mathcal{K}}}
\def\hJ{\widehat{\mathcal{J}}}
\def\hf{\widehat{f}}
\def\<{{\langle}}
\def\>{{\rangle}}
\newfloat{Table}{Hptb}{lot} 

\def\1{{\mathbf 1}}

\def\R{{\mathbb{R}}}

\def\E{{\mathbb{E}}}
\def\P{{\mathbb{P}}}
\def\L{{\mathbb{L}}}
\def\S{{\mathbb{S}}}
\def\hS{\widehat{\mathbb{S}}}

\newcommand{\eref}[1]{(\ref{#1})}
\newcommand{\pa}[1]{\left({#1}\right)}

\newcommand{\cro}[1]{\left[{#1}\right]}

\newcommand{\ac}[1]{\left\{{#1}\right\}}
\newcommand{\norm}[1]{\left\|{#1}\right\|}

\newtheorem{thrm}{Theorem}[section]
\newtheorem{prte}[thrm]{Proposition}
\newtheorem{lemma}[thrm]{Lemma}

\newtheorem{defi}{Definition}[section]

\newcommand{\pen}{\mathrm{pen}}

\newcommand{\crit}{\mathrm{Crit}}
\newcommand{\supp}{\mathrm{supp}}

\newcommand{\AIC}{\textsf{AIC}}
\newcommand{\BIC}{\textsf{BIC}}
\newcommand{\CV}{\textsf{CV}}
\newcommand{\LinSelect}{\textsf{LinSelect}}
\newcommand{\GGMSelect}{\textsf{GGMselect}}
\newcommand{\range}{\mathrm{range}}
\def\cI{\mathcal{I}}

\def\R{\mathbb{R}}
\def\E{\mathbb{E}}

\def\hbeta{\widehat{\beta}}
\def\L{\mathcal{L}}

\def\eps{\varepsilon}

\def\argmin{\mathop{\mathrm{argmin}}}

\begin{document}

\begin{frontmatter}
\title{High-dimensional regression with unknown variance}
\runtitle{Regression with unknown variance}

\begin{aug}
\author{\fnms{Christophe} \snm{Giraud}\ead[label=e1]{
christophe.giraud@polytechnique.edu}},
\author{\fnms{Sylvie}
\snm{Huet}\ead[label=e2]{
sylvie.huet@jouy.inra.fr}}
\and
\author{\fnms{Nicolas}
\snm{Verzelen}\ead[label=e3]{
nicolas.verzelen@supagro.inra.fr}}
\runauthor{Giraud, Huet and Verzelen}
\affiliation{\'Ecole Polytechnique and Institut National de Recherche en Agronomie}
 \address{CMAP, UMR CNRS 7641, Ecole Polytechnique, Route de Saclay, 91128 Palaiseau Cedex, FRANCE. 
 \printead{e1}}
\address{UR341 MIA, INRA, F-78350 Jouy-en-Josas, FRANCE \printead{e2}}
\address{UMR729 MISTEA, INRA, F-34060 Montpellier, FRANCE Montpellier  \printead{e3}}
\end{aug}

\begin{abstract}
We review recent results for  high-dimensional sparse linear regression in the practical case of unknown variance. 
Different sparsity settings are covered, including coordinate-sparsity, group-sparsity and variation-sparsity.
The emphasis is put on non-asymptotic analyses and feasible procedures. In addition, a small numerical study compares the practical performance of three schemes for tuning the Lasso estimator and some references are collected for some more general models, including multivariate regression and nonparametric regression.
\end{abstract}


\end{frontmatter}

\section{Introduction}\label{section_introduction}
In the present paper, we mainly focus on
the linear regression model
\begin{eqnarray}\label{modele1}
 Y={\bf X}\beta_0 +\eps\,,
\end{eqnarray}
where $Y$ is a $n$-dimensional response vector, $\X$ is a fixed $n\times p$ design matrix, and 
the vector $\eps$ is made of $n$ i.i.d\;Gaussian random variables with $\mathcal{N}(0,\sigma^2)$ distribution. In the sequel, $\X^{(i)}$ stands for the $i$-th row of $\X$.
Our interest is on the high-dimensional setting, where the dimension $p$ of  the unknown parameter $\beta_{0}$ is large, possibly larger than $n$.

The analysis of the high-dimensional linear regression model has attracted a lot of attention in the last decade. Nevertheless, 
there is a longstanding gap between the theory where the variance $\sigma^2$ is generally assumed to be known and the practice where it is often unknown. The present paper is mainly devoted to review recent results on linear regression in  high-dimensional settings with \emph{unknown} variance $\sigma^2$.
A few additional results for multivariate regression and the nonparametric regression model
\begin{equation}\label{modele_generale} % Nicolas, tu as des origines italiennes ? 
 Y_i= f(\X^{(i)})+ \eps_i, \quad\quad i=1,\ldots,n\,,
\end{equation}
 will also be mentioned.

\subsection{Sparsity assumptions}
In a high-dimensional linear regression model, accurate estimation  is unfeasible unless it relies on some special properties of the parameter $\beta_0$. The most common assumption on  $\beta_0$ is  that it is sparse in some sense. 
We will consider in this paper the three following classical sparsity assumptions.\smallskip

\noindent
\textsf{\textbf{Coordinate-sparsity.}} Most of the coordinates of
$\beta_0$ are assumed to be  zero (or approximately zero). This is the
most common acceptation for sparsity in linear
regression.

\smallskip

\noindent 
\textsf{\textbf{Structured-sparsity.}} The pattern of zero(s) of the coordinates  of $\beta_0$ is assumed to have an a priori known structure. For instance, in group-sparsity~\cite{groupyuan}, the covariates  are clustered into $M$ groups and when the coefficient $\beta_{0,i}$ corresponding to the covariate ${\bf X}_i$ (the $i$-th column of ${\bf X}$) is non-zero, then it is likely that all the coefficients $\beta_{0,j}$ with variables ${\bf X}_j$ in the same cluster as ${\bf X}_i$ are non-zero.\smallskip

\noindent 
\textsf{\textbf{Variation-sparsity.}} The $p-1$-dimensional vector $\beta^V_0$ of variation of $\beta_0$ is defined by $\beta^V_{0,j}=\beta_{0,j+1}-\beta_{0,j}$. Sparsity in variation means that most of the components of $\beta^V_0$ are equal to zero (or approximately zero). When $p=n$ and ${\bf X}=I_n$, variation-sparse linear regression  corresponds to signal segmentation.

\subsection{Statistical objectives}

In the linear regression model, there are roughly two kinds of estimation objectives. In the \emph{prediction problem}, the goal is to estimate ${\bf X}\beta_0$, whereas in the \emph{inverse problem} it is to  estimate $\beta_0$. When the vector $\beta_0$ is sparse, a related objective is to estimate the \emph{support} of $\beta_0$ (model identification problem) which is the set of the indices $j$ corresponding to the non zero coefficients $\beta_{0,j}$.
Inverse problems and prediction problems are not equivalent in general. When the Gram matrix ${\bf X}{\bf X}^*$ is poorly conditioned, the former problems can be much more difficult than the latter. Since there are only a few results on inverse problems with unknown  variance, we will focus on the prediction problem, the support estimation problem being shortly discussed in the course of the paper.  

In the sequel,  $\mathbb{E}_{\beta_0}[.]$ stands for the expectation with respect to $Y\sim\mathcal{N}({\bf X}\beta_0,\sigma^2 I_n)$ and $\|.\|_2$  is the euclidean norm. The prediction objective amounts  to build estimators $\widehat{\beta}$ so that the risk 
\begin{equation}\label{definition_risque_prediction}
 \mathcal{R}[\widehat{\beta};\beta_0]:= \mathbb{E}_{\beta_0}[\|{\bf X}
(\widehat{\beta}-\beta_0)\|_2^2]
\end{equation}
 is as small as possible. 

\subsection{Approaches}

Most procedures that handle high-dimensional linear models~\cite{candes07,DT08,tsyrig10,tiblasso,2005_TSR_JRSSB,zhangMC+,zhangFwdBwd,zou05} rely on tuning parameters whose 
optimal value depends on $\sigma$. For example, the results of Bickel et al.~\cite{bickeltsy08}
% state that under some assumptions on ${\bf X}$, 
suggest to choose the tuning parameter $\lambda$ of the Lasso 
%should be chosen 
of the order of $2\sigma\sqrt{2\log(p)}$. As a consequence, all these procedures cannot be directly applied when $\sigma^2$ is unknown.

A straightforward approach is to replace $\sigma^2$ by an estimate of the variance in the optimal value of the tuning parameter(s). Nevertheless,  the variance $\sigma^2$ is difficult to estimate in  high-dimensional settings, so  a plug-in of the variance does not necessarily yield good results. There are basically two approaches to build on this amount of work on high-dimensional estimation with known variance.
\begin{enumerate}
\item \textsf{\textbf{Ad-hoc estimation.}} There has been some recent
  work~\cite{squarerootlasso,stadler,sun} to modify procedures like
  the Lasso in such a way that the tuning parameter does not depend
  anymore on $\sigma^2$ (see Section \ref{section_tune_lasso}). The
  challenge is to find a smart modification of the procedure, so that
  the resulting estimator $\widehat{\beta}$ is computationally
  feasible and has a risk
  $\mathcal{R}\left[\widehat{\beta};\beta_0\right]$ as small as
  possible. 

 \item \textsf{\textbf{Estimator selection.}} Given a collection $(\widehat{\beta}_{\lambda})_{\lambda\in\Lambda}$ of estimators, the objective of estimator selection is to pick an index $\widehat{\lambda}$ such that the risk of $\widehat{\beta}_{\widehat{\lambda}}$ is as small as possible; ideally as small as the 
 risk $\mathcal{R}[\widehat{\beta}_{\lambda^*};\beta_0]$ of the so-called \emph{oracle}  estimator
 %$\widehat \beta_{\lambda^*}$ defined by
\begin{equation}\label{tuning_optimal}
\widehat\beta_{\lambda^*}:= \argmin_{\{\widehat\beta_{\lambda},\ \lambda\in \Lambda\}} \mathcal{R}\left[\widehat{\beta}_{\lambda};\beta_0\right]\ .
 \end{equation}
Efficient estimator selection procedures can then be applied to tune  the
aforementioned estimation
methods~\cite{candes07,DT08,tsyrig10,tiblasso,2005_TSR_JRSSB,zhangMC+,zhangFwdBwd,zou05}. Among the most famous methods for estimator selection, we mention $V$-fold cross-validation (Geisser~\cite{Geisser75}), \AIC\ (Akaike~\cite{Akaike73}) and \BIC\ (Schwarz~\cite{Schwartz78}) criteria.
\end{enumerate}
The objective of this survey is to describe state-of-the-art procedures for  high-dimensional linear regression with unknown variance. 
We will review both automatic tuning methods and  ad-hoc methods.
There are some procedures that we will let aside. For example, Baraud~\cite{Baraud10} provides a versatile estimator selection scheme, but the procedure is computationally intractable  in large dimensions.
Linear or convex aggregation of estimators are also valuable alternatives to estimator selection when the goal is to perform \emph{estimation}, but only a few theoretical works have addressed the aggregation problem when the variance is unknown~\cite{giraudagregation,sebastien_agregation}. For these reasons, we will not review these approaches in the sequel.

\subsection{Why care about non-asymptotic analyses ?}
\AIC~\cite{Akaike73}, \BIC~\cite{Schwartz78} and $V$-fold Cross-Validation~\cite{Geisser75} are probably the most popular criteria for estimator selection. The use of these criteria relies on some classical asymptotic optimality results. These results focus on the setting where the collection of estimators $(\hbeta_{\lambda})_{\lambda\in\Lambda}$ and the dimension $p$ are fixed and consider the limit behavior of the criteria when the sample size $n$ goes to infinity. For example, under some suitable conditions, Shibata~\cite{Shibata81}, Li~\cite{Li87} and Shao~\cite{Shao97} prove that the risk of the estimator selected by \AIC\ or $V$-fold \CV\ (with $V=V_{n}\to \infty$) is asymptotically equivalent to the oracle risk $\mathcal{R}[\widehat{\beta}_{\lambda^*};\beta_0]$. Similarly, 
Nishii~\cite{Nishii84} shows that the \BIC\ criterion is consistent for model selection.
%All the above-mentioned classical results are of asymptotic nature. 

All these asymptotic results can lead to misleading conclusions in  modern statistical settings where the sample size remains small and the parameter's dimension becomes large. For instance it is proved in \cite[Sect.3.3.2]{BGH09} and illustrated in \cite[Sect.6.2]{BGH09}
that \BIC\ (and thus \AIC) can strongly overfit and should not be used for $p$ larger than $n$. Additional examples are provided in Appendix \ref{section_appendix_BIC}.
 A non-asymptotic analysis takes into account  all the characteristics of the selection problem (sample size $n$, parameter dimension $p$, number of models per dimension, etc). It treats $n$ and $p$ as they are and it avoids to miss important features hidden in asymptotic limits. 
 %It also excludes the uncomfortable assumptions required for the above-mentioned asymptotic results.
 %$\beta\notin\bigcup_{\lambda\in\Lambda}S_{\lambda}$ required for \AIC\ or $\beta\in S_{\lambda_{0}}$ required for \BIC. In the following, we recall that we focus on a prediction point of view, that is the goal of a selection scheme is to select an estimator $\widehat{\beta}$ with minimal risk $\mathbb{E}_{\beta_0}[\|{\bf X}(\widehat{\beta}-\beta_0)\|_2^2]$.
For these reasons, we will restrict in this review on non-asymptotic results.

\subsection{Organization of the paper}
In Section \ref{section_limites_theorique}, we investigate how the ignorance of the variance affects the minimax risk bounds. In Section~\ref{generic_schemes.st}, some "generic" estimators selection schemes are presented. The coordinate-sparse setting is addressed  Section~\ref{section_univariate} : some theoretical results are collected and a small numerical experiment compares different Lasso-based procedures.
The group-sparse and variation-sparse settings are reviewed in Section~\ref{subsection_groupe} and~\ref{section_rupture}, and   Section~\ref{section_extension} is devoted to some more general models such as multivariate regression or nonparametric regression.\smallskip

In the sequel, $C$, $C_1$,\ldots  refer to numerical constants whose value may vary from line to line, while $\|\beta\|_0$ stands for the number of non zero components of $\beta$ and $|\mathcal{J}|$ for the cardinality of a set $\mathcal{J}$.

\section{Theoretical limits}\label{section_limites_theorique}
The goal of this section is to address the intrinsic difficulty of a coordinate-sparse linear regression  problem. We will answer the following questions: 
Which range of $p$ can we reasonably consider? When the variance is unknown, can we hope to do as well as when the variance is known?

\subsection{Minimax adaptation}
A classical way to assess the performance of an estimator $\widehat{\beta}$ is
to  measure its maximal risk over a class $\boldsymbol{B}\subset \mathbb{R}^p$. This is
the minimax point of view. As we are interested in coordinate-sparsity for $\beta_0$, we will consider the sets $\boldsymbol{B}[k,p]$ of vectors 
that contain at most $k$ non zero coordinates for some $k>0$. 

Given an estimator $\widehat{\beta}$, the {\it maximal prediction
risk} of $\widehat{\beta}$ over $\boldsymbol{B}[k,p]$ for a fixed design ${\bf X}$ and a
variance $\sigma^2$ is defined by
$\sup_{\beta_0\in\boldsymbol{B}[k,p]}\mathcal{R}[\widehat{\beta};\beta_0]$ where the risk function $\mathcal{R}[.,\beta_0]$ is defined by (\ref{definition_risque_prediction}). Taking the infimum of the maximal risk over all
possible estimators $\widehat{\beta}$, we obtain the
{\it minimax risk}
\begin{equation}
\label{definition_minimax_prediction_fixe}
\mathbf{R}[k,{\bf
X}]=\inf_{\widehat{\beta}}\sup_{\beta_0\in\boldsymbol{B}[k,p]}\mathcal{R}[\widehat{\beta};\beta_0]\ .
\end{equation}
Minimax bounds are convenient results to assess the range of problems that are statistically feasible and the optimality of particular procedures. Below, we say 
that an estimator
$\widehat{\beta}$
is ``minimax'' over $\boldsymbol{B}[k,p]$ if its maximal prediction risk is close to the minimax risk. 

In practice, the number of non-zero coordinates of $\beta_0$ is unknown. The fact that an estimator $\widehat{\beta}$ is minimax 
over $\boldsymbol{B}[k,p]$ for some specific $k>0$ does not imply that $\widehat{\beta}$ estimates well vectors $\beta_0$ that are less sparse. A good estimation procedure $\widehat{\beta}$ should not require the knowledge of the sparsity $k$ of $\beta_0$ and should perform as well as if this sparsity were known. An estimator $\widehat{\beta}$ that nearly achieves the minimax risk over
$\boldsymbol{B}[k,p]$ for a range of $k$  is  said to be {\it adaptive} to
the sparsity. Similarly,  an estimator $\widehat{\beta}$ 
is adaptive to the variance $\sigma^2$, if it does not require the
knowledge of $\sigma^2$ and nearly achieves the minimax risk for all 
$\sigma^2>0$. When possible, the main challenge is to build adaptive procedures.
% For some statistical problems, adaptation is in fact impossible and there is an unavoidable loss when the variance or the sparsity parameter is unknown. In such situations, it is interesting to quantify this loss.
\smallskip

In the following subsections, we review sharp bounds on the minimax prediction risks for both  known and unknown sparsity, known and unknown variance. The big picture is summed up in Figure \ref{figure_minimax}. Roughly, it says that adaptation is possible as long as $2k\log(p/k)< n$. In contrast, the situation becomes more complex for  the ultra-high-dimensional\footnote{In some papers, the expression ultra-high-dimensional has been used to characterize problems such that $\log(p)=O(n^{\theta})$ with $\theta<1$. We argue here that as soon as $k\log(p)/n$ goes to $0$, the case  $\log(p)=O(n^{\theta})$ is not intrinsically more difficult than conditions such as $p=O(n^{\delta})$ with $\delta>0$.} setting where $2k\log(p/k)\geq  n$. The rest of this section is devoted to explain this big picture.
\begin{figure}[hptb]
\begin{center}
{\includegraphics[width=9cm,angle=0]{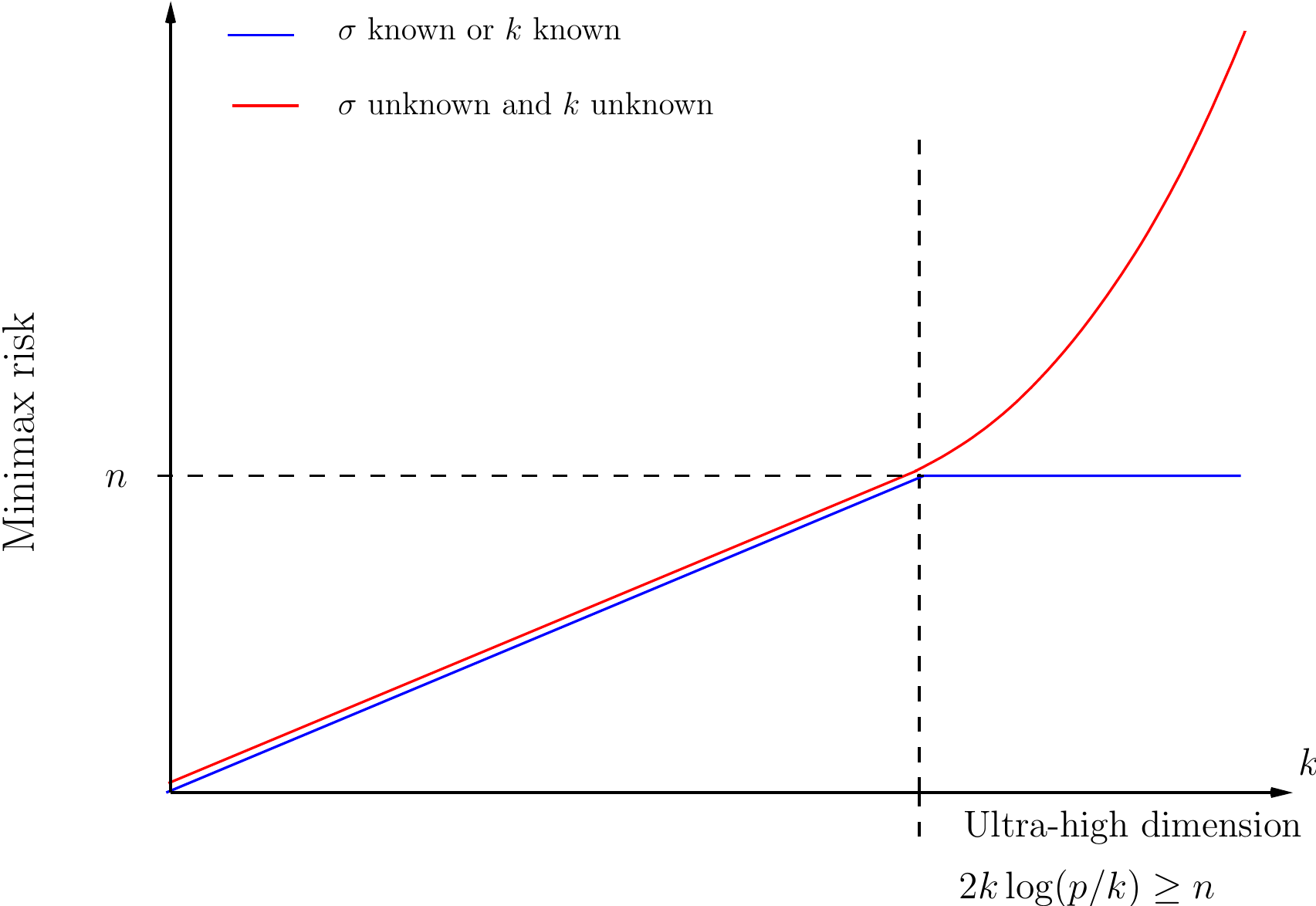}}
\caption{Minimal prediction risk over $\boldsymbol{B}[k,p]$ as a function of $k$. }\label{figure_minimax} 
\end{center}
\end{figure}

\subsection{Minimax risks under known sparsity and known variance}

The minimax risk $\mathbf{R}[k,{\bf X}]$ depends on the form of the design ${\bf X}$. In order to grasp this dependency, we define for any $k>0$,  the largest and the smallest sparse eigenvalues of order $k$ of ${\bf
X}^*{\bf X}$ by 
\begin{equation*}
\varPhi_{k,+}({\bf X}) :=
\sup_{\beta\in \boldsymbol{B}[k,p]\setminus\{0_p\}}\frac{\|{\bf
X}\beta\|_n^2}{\|\beta\|_p^2}\hspace{0.5cm}\text{ and }\hspace{0.5cm}
\varPhi_{k,-}({\bf X}) :=\inf_{\beta\in\boldsymbol{B}[k,p]\setminus\{0_p\}}\frac{\|{\bf
X}\beta\|_n^2}{\|\beta\|_p^2}\ .
\end{equation*}
.

\begin{prte}\label{prte_design} Assume that $k$ and $\sigma$ are known.
There exist positive numerical constants $C_1$, $C'_1$, $C_2$, and $C'_2$ such that the following holds.
For any $(k,n,p)$ such that $k\leq n/2$ and any design ${\bf X}$, we have
\begin{equation}\label{minoration_minimax_prediction_design}
C_1 \frac{\varPhi_{2k,-}({\bf X})}{\varPhi_{2k,+}({\bf X})}
k\log\left(\frac{p}{k}\right)\sigma^2\leq 
 \mathbf{R}[k,{\bf X}]\leq C'_1 \left[k\log\left(\frac{p}{k}\right)\wedge n\right]\sigma^2\ ,
\end{equation}
For any $(k,n,p)$ such that $k\leq n/2$, we have
\begin{equation}\label{minoration_minimax_prediction}
 C_2 
\left[k\log\left(\frac{p}{k}\right)\wedge n
\right]\sigma^2\leq \sup_{{\bf X}}\mathbf{R}[k,{\bf X}]\leq
C'_2 \left[k\log\left(\frac{p}{k}\right)\wedge n
\right]\sigma^2\ .
\end{equation}
\end{prte}
The minimax lower bound (\ref{minoration_minimax_prediction_design}) has been first proved in \cite{raskwain09,tsyrig10,yeminimax} while  (\ref{minoration_minimax_prediction}) is stated in \cite{Vminimax}. Let us first comment the bound (\ref{minoration_minimax_prediction}). If the vector $\beta_0$ has $k$-non zero components and if these components are {\it a priori} known, then one may build estimators that achieve a risk bound of the order $k$. In a (non-ultra) high-dimensional setting ($2k\log(p/k)\leq n$), the minimax risk is of the order $k\log(p/k)\sigma^2$. The logarithmic term is the price to pay to cope with the fact that we do not know the position of the non zero components in $\beta_0$. The situation  is quite different in an ultra-high-dimensional setting ($2k\log(p/k)> n$). Indeed, the minimax risk remains of the order of $n\sigma^2$, which corresponds to the minimax risk of estimation of the vector ${\bf X}\beta_0$ without any sparsity assumption (see the blue curve in Figure \ref{figure_minimax}). In other terms, the sparsity index $k$ does not play a role anymore. 

\medskip

\noindent\textsc{\bf Dependency of $\mathbf{R}[k,{\bf X}]$ on the design ${\bf X }$}. It follows from (\ref{minoration_minimax_prediction_design}) that $\sup_{{\bf X}}\mathbf{R}[k,{\bf X}]$ is nearly achieved by designs ${\bf X}$ satisfying $\varPhi_{2k,-}({\bf X})/\varPhi_{2k,+}({\bf X})\approx 1$, when the setting is not ultra-high dimensional. 
For some designs such that $ \varPhi_{2k,-}({\bf X})/\varPhi_{2k,+}({\bf X})$ is small, the minimax prediction risk $\mathbf{R}[k,{\bf X}]$ is possibly faster (see \cite{Vminimax} for a discussion). In a ultra-high dimensional, the form of the minimax risk ($n\sigma^2$) is related to the fact that no designs can satisfy $ \varPhi_{2k,-}({\bf X})/\varPhi_{2k,+}({\bf X})\approx 1$ (see e.g.~\cite{baraniuk08}). The lower bound $\mathbf{R}[k,{\bf X}]\geq C \left[k\log(p/k)\wedge n\right]
\sigma^2$ in (\ref{minoration_minimax_prediction}) is for instance achieved by realizations of a standard Gaussian design, that is designs ${\bf X}$ whose components follow independent standard normal distributions.
See~\cite{Vminimax} for more details.\\

\subsection{Adaptation to the sparsity and to the variance}

\noindent\textsc{\bf Adaptation to the sparsity when the variance is known}. When $\sigma^2$ is known, there exist both model selection and aggregation procedures that achieve this 
$[k\log(p/k)\wedge n]\sigma^2$ risk simultaneously for all $k$ and for
all designs ${\bf X}$. Such procedures derive from
the work of Birg\'e and Massart~\cite{BM01} and Leung and
Barron~\cite{LB06}. However, these methods are
intractable for large $p$ except for specific forms of the design. We refer to Appendix
\ref{section_appendix_minimax_1} for more details. 
\medskip

\noindent\textsc{\bf Simultaneous adaptation to the sparsity and the variance}. We first restrict  to the non-ultra high-dimensional setting,  where the number of non-zero components $k$ is unknown but satisfies $2k\log(p/k)< n$. In this setting, some  procedures based on penalized log-likelihood~\cite{BGH09} are simultaneous adaptive to the unknown sparsity and to the unknown variance and this for all designs ${\bf X}$. Again such procedures are intractable for large $p$. See Appendix \ref{section_appendix_minimax_2} for more details. If we want to cover all $k$ (including ultra-high dimensional settings), the situation is different as shown in the next proposition (from~\cite{Vminimax}).

 \begin{prte}[Simultaneous adaptation is impossible]\label{prte_adaptation_impossible}
There exist positive constants $C$, $C'$, $C_1$, $C_2$, $C_3$, $C'_1$, $C'_2$, and $C'_3$, such that the following holds.
Consider any $p\geq n\geq C$ and $k\leq p^{1/3}\wedge n/2$ such that
$k\log(p/k)\geq C' n$. There exist designs
${\bf X}$ of size $n\times p$ such that
for any estimator $\widehat{\beta}$, we have either
\begin{eqnarray*}
\sup_{\
\sigma^2>0}\frac{\mathcal{R}[\widehat{\beta};0_p]}{\sigma^2} &> &C_1n\ , \hspace{2cm}\text{or}\\
\sup_{\beta_0\in\boldsymbol{B}[k,p]\
,\ \sigma^2>0}\frac{\mathcal{R}[\widehat{\beta};\beta_0]}{\sigma^2}
& > &C_2k\log\left(\frac{p}{k}\right)\exp\left[C_3\frac{k}{n}\log\left(\frac{p}{k}\right)\right]\ .
\end{eqnarray*}
 Conversely, there exist two estimators $\widehat{\beta}^{(n)}$ and $\widehat{\beta}^{BGH}$ (defined in Appendix \ref{section_appendix_minimax_2}) that respectively satisfy~\\
\begin{eqnarray*}
\sup_{\bf X}\sup_{\beta_0\in\mathbb{R}^p,\ 
\sigma^2>0}\frac{\mathcal{R}[\widehat{\beta}^{(n)};\beta_0]}{\sigma^2} &\leq& C'_1n\ ,\\ 
\sup_{\bf X}\sup_{\beta_0\in\boldsymbol{B}[k,p]\
,\ \sigma^2>0}\frac{\mathcal{R}[\widehat{\beta}^{BGH};\beta_0]}{\sigma^2}
 &\leq & C'_2k\log\left(\frac{p}{k}\right)\exp\left[C'_3\frac{k}{n}\log\left(\frac{p}{k}\right)\right]\ , 
\end{eqnarray*}
for all $1\leq k\leq [(n-1)\wedge p]/4$.

\end{prte}
 As a consequence, simultaneous adaptation to the sparsity and to the variance is impossible in an ultra-high dimensional setting. Indeed, any estimator $\widehat{\beta}$ that does not rely on $\sigma^2$ has to pay at least one of
these two prices:
\begin{enumerate}
\item The estimator $\widehat{\beta}$ does not use the sparsity of the true
parameter $\beta_0$ and its risk for estimating ${\bf X}0_p$ is of the same order as the
minimax risk over $\mathbb{R}^n$. 
\item For any $1\leq k\leq p^{1/3}$, the risk of $\hbeta$ fulfills
\begin{eqnarray*}
\sup_{\sigma>0}\sup_{\beta_0\in\boldsymbol{B}[k,p]}\frac{\mathcal{R}[\widehat{\beta};\beta_0]}{\sigma^2}
\geq
C_1k\log\left(p\right)\exp\left[C_2\frac{k}{n}
\log\left(p\right)\right ]\ .
\end{eqnarray*}
It follows that the maximal risk of $\widehat{\beta}$ is blowing up in an ultra-high-dimensional setting (red curve in Figure \ref{figure_minimax}), while the minimax risk is stuck to $n$ (blue curve in  Figure \ref{figure_minimax}). The designs that satisfy the minimax lower bounds of Proposition \ref{prte_adaptation_impossible} include realizations of a standard Gaussian design.
\end{enumerate}

%The previous results have illustrated the existence of a phase transition when $k$ and $p$ are very  large ($2k\log(p/k)\geq n$): the prediction problem with unknown %variance and unknown sparsity becomes extremely difficult in ultra-high-dimensional settings.

In an ultra-high dimensional setting, the prediction problem becomes extremely difficult under unknown variance because the variance estimation itself is inconsistent as shown in the next proposition (from~\cite{Vminimax}). 
\begin{prte} There exist positive constants $C$, $C_1$, and $C_2$ such that the following holds.
Assume that $p\geq n\geq C$.
For any $1\leq k\leq p^{1/3}$, there exist designs ${\bf X}$ such that 
\begin{equation*}
\inf_{\widehat{\sigma}} \sup_{\sigma>0,\ \beta_0\in\mathbf{B}[k,p]}\mathbb{E}_{\beta_0}\left[\left|\frac{\sigma^2}{\widehat{\sigma}^2}-\frac{\widehat{\sigma}^2}{\sigma^2}\right|\right]\geq C_1\frac{k}{n}\log\left(\frac{p}{k}\right)\exp\left[C_2\frac{k}{n}
\log\left(\frac{p}{k}\right)\right ]\ .
\end{equation*}

\end{prte}

\smallskip

\subsection{What should we expect from a good estimation procedure?}\label{section_limites_theorique_object}

Let us consider an estimator $\widehat{\beta}$ that does not depend on $\sigma^2$. Relying on the previous minimax bounds, we will say that $\widehat{\beta}$ achieves an {\it optimal} risk bound  (with respect to the sparsity) if 
\begin{equation}\label{borne_optimales_risque}
 \mathcal{R}[\widehat{\beta};\beta_0]\leq C_1 \|\beta_0\|_0\log(p)\sigma^2\ ,
\end{equation}
for any $\sigma>0$ and any vector $\beta_0\in\mathbb{R}^p$ such that $1\leq \|\beta_0\|_0\log(p)\leq C_2 n$. Such risk bounds prove that the estimator is approximately (up to a possible $\log(\|\beta_{0}\|_{0})$ additional term) minimax adaptive to the unknown variance and the unknown sparsity. The condition  $\|\beta_0\|_0\log(p)\leq C_2 n$ ensures that the setting is not ultra-high-dimensional.  As stated above, some procedures achieve (\ref{borne_optimales_risque}) for all designs ${\bf X}$ but they are intractable for large $p$ (see Appendix \ref{section_appendix_minimax}). One purpose of this review is to present fast procedures that achieve this  kind of bounds  under  possible restrictive assumptions on the design matrix ${\bf X}$.

For some procedures, (\ref{borne_optimales_risque}) can be improved into a bound of the form
\begin{equation}\label{bel-oracle}
 \mathcal{R}[\widehat{\beta};\beta_0]\leq C_1\inf_{\beta\neq 0}\left\{\|{\bf X}(\beta-\beta_{0})\|_2^2+ \|\beta\|_0 \log(p)\sigma^2\right\}\ ,
\end{equation}
with $C_1$ close to one. Again, the dimension $\|\beta_0\|_0$ is restricted to be smaller than $Cn/\log(p)$ to ensure that the setting is not ultra-high dimensional. This kind of bound makes a clear trade-off between a bias and a variance term. 
For instance, when $\beta_0$ contains many components that are nearly equal to zero, the  bound~\eref{bel-oracle} can be much smaller than (\ref{borne_optimales_risque}).

\subsection{Other statistical problems in an ultra-high-dimensional setting}
We have seen that adaptation becomes impossible for the prediction problem in a ultra-high dimensional setting. For other statistical problems, including the prediction problem with  random design, the inverse problem (estimation of $\beta_0$), the variable selection problem (estimation of the support of $\beta_0$), the dimension reduction problem~\cite{Vminimax,wain_minimax2,jinups}, the minimax risks are blowing up in a ultra-high dimensional setting. This kind of phase transition has been observed in a wide range of random geometry problems~\cite{donoho_transition}, suggesting some universality in this limitation.
%It follows from the minimax lower bounds that the ultra-high dimensional setting does not allows an accurate estimation.
%In practice, where lie the limits of accurate high-dimensional sparse estimation? 
In practice, the sparsity index $k$ is not known, but given $(n,p)$ we can compute $k^*:=\max\{k:\ 2k\log(p/k)\geq n\}$. One may interpret that the problem is still reasonably difficult as long as $k\leq k^*$. This gives a simple rule of thumb to know what we can hope from a given regression problem. For example, setting $p=5000$ and $n=50$ leads to $k^*=3$, implying that the prediction problem becomes extremely difficult when there are more than 4 relevant covariates (see the simulations in~\cite{Vminimax}).

\section{Some generic selection schemes}\label{generic_schemes.st}
Among the selection schemes not requiring the knowledge of  the variance $\sigma^2$, some are very specific to a particular algorithm, while some others are more generic. We describe in this section three versatile selection principles and refer to the examples for the more specific schemes. 

\subsection{Cross-Validation procedures}
The cross-validation schemes are nearly universal in the sense that they can be implemented in most statistical frameworks and for most  estimation procedures. 
The principle of the cross-validation schemes is to split the data into a \emph{training} set and a \emph{validation} set : the estimators are built on the \emph{training} set and the \emph{validation} set is used for estimating their prediction risk. This training / validation splitting  is eventually repeated several times.
The most popular cross-validation schemes are :
\begin{itemize}
\item \emph{Hold-out}~\cite{MostellerTukey68, DevroyeWagner79} which is based on a single split of the data for \emph{training} and \emph{validation}.
\item \emph{$V$-fold} \CV\ \cite{Geisser75}. The data is split into $V$ subsamples. Each subsample is successively removed for \emph{validation}, the remaining data being used for \emph{training}.
\item \emph{Leave-one-out}~\cite{Stone74} which corresponds to $n$-fold \CV.
\item \emph{Leave-$q$-out} (also called \emph{delete-$q$-CV})~\cite{Shao93} where every possible
 subset of cardinality $q$ of the data  is removed for \emph{validation}, the remaining data being used for \emph{training}. 
\end{itemize}
We refer to Arlot and C\'elisse~\cite{ArlotCelisse} for a  review of the cross-validation schemes and their theoretical properties.

\subsection{Penalized empirical loss}
Penalized empirical loss criteria form another class of versatile selection schemes, yet less universal than \CV\ procedures. 
The principle is to select among a family $(\hbeta_{\lambda})_{\lambda\in\Lambda}$ of estimators by minimizing a criterion of the generic form
\begin{equation}\label{penalized-loss}
\crit(\lambda)=\L_{\bf X}(Y,\hbeta_{\lambda})+\pen(\lambda),
\end{equation}
where $\L_{\bf X}(Y,\hbeta_{\lambda})$ is a measure of the distance between $Y$ and ${\bf X}\hbeta_{\lambda}$, and
 $\pen$ is a function from $\Lambda$ to $\R^+$. The penalty function sometimes  depends on data. \smallskip
 
\noindent\textsf{\textbf{Penalized log-likelihood.}} The most famous criteria of the form~\eref{penalized-loss} are \AIC\ and \BIC. They have been designed to select among estimators $\hbeta_{\lambda}$ obtained by maximizing the likelihood of $(\beta,\sigma)$ with the constraint that $\beta$ lies on a linear space $S_{\lambda}$ (called \emph{model}). In the Gaussian case, these estimators are given by $\X\hbeta_{\lambda}=\Pi_{S_{\lambda}}Y$, where $\Pi_{S_{\lambda}}$ denotes the orthogonal projector onto the model $S_{\lambda}$.
For \AIC\ and \BIC, the function $\L_{{\bf X}}$ corresponds to twice the negative log-likelihood  $\L_{{\bf X}}(Y,\hbeta_{\lambda})=n\log(\|Y-\X\hbeta_{\lambda}\|_2^2)$ and the penalties are $\pen(\lambda)=2\dim(S_{\lambda})$  and $\pen(\lambda)=\dim(S_{\lambda})\log(n)$ respectively. We recall that these two criteria can perform very poorly in a high-dimensional setting.
 
In the same setting, Baraud {\it et al.}~\cite{BGH09} propose alternative penalties built from a non-asymptotic perspective. 
The resulting criterion can handle the high-dimensional setting where $p$ is possibly larger than $n$ and the risk of the selection procedure is controlled by a bound of the form~(\ref{bel-oracle}), see Theorem~2 in~\cite{BGH09}. \smallskip

\noindent\textsf{\textbf{Plug-in criteria.}} 
Many other penalized-empirical-loss criteria have been developed in the last decades. Several selection criteria~\cite{BaronBirgeMassart99, BM01} have been designed from a non-asymptotic point of view  to handle the case where the variance is known. 
These criteria usually involve the residual least-square $\L_{{\bf X}}(Y,\hbeta_{\lambda})=\|Y-\X\hbeta_{\lambda}\|_2^2$ and a
 penalty $\pen(\lambda)$  depending on the variance $\sigma^2$. A common practice is then to plug in the penalty an estimate $\hat \sigma^2$ of the variance in place of the variance. For linear regression, when the  design matrix $\X$ has a rank less than $n$, a classical choice for $\hat \sigma^2$ is
$$\hat\sigma^2={\|Y-\Pi_{\X} Y\|_2^2\over n-\mathrm{rank}(\X)}\,,$$
with $\Pi_{\X}$ the orthogonal projector onto the range of $\X$. 
This estimator $\hat\sigma^2$ has the nice feature to be independent of $\Pi_{\X} Y$ on which usually rely the estimators $\hbeta_{\lambda}$. Nevertheless, the variance of $\hat\sigma^2$ is of order $\sigma^4/\pa{n-\mathrm{rank}(\X)}$ which is small only when the sample size $n$ is quite large in front of the rank of $\X$. This situation is unfortunately not likely to happen in a high-dimensional setting where $p$ can be larger than $n$.

\subsection{Approximation versus complexity penalization : \LinSelect}
The criterion proposed by
Baraud {\it et al.}~\cite{BGH09} can handle high-dimensional settings but it suffers from two rigidities. First, it can only handle \emph{fixed} collections of models $(S_{\lambda})_{\lambda\in\Lambda}$. In some situations, the size of $\Lambda$ is huge (e.g. for complete variable selection) and  the estimation procedure can then be  computationally intractable. In this case, we may want to work with a subcollection 
of models $(S_{\lambda})_{\lambda\in\widehat\Lambda}$, where $\widehat\Lambda\subset\Lambda$ may depend on data. For example, for  complete variable selection, the subset $\widehat\Lambda$ could be generated by efficient algorithms like LARS~\cite{lars}.
The second rigidity of the procedure of Baraud {\it et al.}~\cite{BGH09} is that it can only handle constrained-maximum-likelihood estimators. This procedure then does not help for selecting among arbitrary estimators such as the Lasso or Elastic-Net.

These two rigidities have been addressed recently by Baraud {\it et al.}~\cite{linselect}. They propose a selection procedure, \LinSelect, which can handle both data-dependent collections of models and arbitrary estimators $\hbeta_{\lambda}$. 
The procedure is based on a collection $\S$ of linear spaces which gives  a collection of possible "approximative" supports for the estimators $(\X\hbeta_{\lambda})_{\lambda\in\Lambda}$. A measure of  complexity on $\S$
is provided by a weight function $\Delta: \S \to \R^+$ . We refer to Sections~\ref{tuning-lasso} and~\ref{subsection_groupe} for examples of collection $\S$ and weight $\Delta$ in the context of coordinate-sparse and group-sparse regression.
We present below a simplified version of the \LinSelect\ procedure.
For a suitable, possibly data-dependent, subset $\hS\subset \S$ (depending on the statistical problem), the estimator $\hbeta_{\widehat \lambda}$ is selected by minimizing the criterion
\begin{equation}\label{eq:linselect}
\crit(\hbeta_{\lambda})=\inf_{S\in\hS}\cro{
\|Y-\Pi_{S}\X\hbeta_{\lambda}\|_2^{2}+{1\over 2}\|\X\hbeta_{\lambda}-\Pi_{S}\X\hbeta_{\lambda}\|_2^{2}+\pen(S)\,\widehat\sigma^2_{S}},
\end{equation}
where $\Pi_{S}$ is the orthogonal projector onto  $S$,  
\begin{equation*}\label{eq:var}
\widehat \sigma^2_{S}={\norm{Y-\Pi_{S}Y}^2_{2}\over n-\dim(S)},
\end{equation*}
and $\pen(S)$ is a penalty depending on $\Delta$.
In the cases we will consider here, the penalty $\pen(S)$ is roughly of the order of $\Delta(S)$ and therefore it penalizes $S$ according to its complexity.
We refer to the Appendix~\ref{appendix-linselect}  for a precise definition of this penalty and more details on its characteristics. 
 We emphasize  that the Criterion~(\ref{eq:linselect}) and the family of estimators $\{\hbeta_{\lambda},\ \lambda\in\Lambda\}$ are based on 
 the \emph{same} data $Y$ and $\X$. In other words, there is no data-splitting occurring in the \LinSelect\ procedure. The first term in~(\ref{eq:linselect}) quantifies the fit of the projected estimator to the data, the second term evaluates the approximation quality of the space $S$ and the last term penalizes $S$ according to its complexity.
 We refer to Proposition~\ref{ORACLE-LINSELECT} in Appendix~\ref{appendix-linselect}
 and Theorem~1 in~\cite{BGH09}  for  risk bounds on the selected estimator.
 Instantiations of the procedure and more specific risks bounds are given 
in Sections~\ref{section_univariate} and~\ref{subsection_groupe} in the context of coordinate-sparsity and group-sparsity. 

From a computational point of view, 
the algorithmic complexity of \LinSelect\ is at most proportional to $|\Lambda|\times|\hS|$ and in many cases there is 
no need to scan the whole set $\hS$ for each $\lambda\in\Lambda$ to minimize~(\ref{eq:linselect}).
In the examples of  Sections~\ref{section_univariate} and~\ref{subsection_groupe}, the whole procedure is computationally less intensive than $V$-fold \CV, see Table~\ref{CompTime.tb}.
Finally, we mention that for the constrained least-square estimators  $\X\hbeta_{\lambda}=\Pi_{S_{\lambda}}Y$, the \LinSelect\ procedure with $\hS=\ac{S_{\lambda}:\lambda\in\Lambda}$ simply coincides with the procedure of Baraud {\it et al.}~\cite{BGH09}.

\section{Coordinate-sparsity}\label{section_univariate}
In this section, we focus on the high-dimensional linear regression model $Y={\bf X}\beta_0+\varepsilon$ where the vector $\beta_0$ itself is assumed to be sparse. 
This setting has attracted a lot of attention in the last decade, and many  estimation procedures have been developed. Most of them require the choice of tuning parameters which depend on the unknown variance~$\sigma^2$. This is for instance the case for the Lasso~\cite{tiblasso,chen98}, Dantzig Selector~\cite{candes07}, Elastic Net~\cite{zou05}, MC+~\cite{zhangMC+}, aggregation techniques~\cite{tsybakov_agregation_07,DT08}, etc. 

We first discuss how the generic schemes  introduced in the previous section can be instantiated for tuning these procedures and for selecting among them. Then, we pay a special attention to  the calibration of the Lasso. Finally, we discuss the problem of support estimation and present a small numerical study.

\subsection{Automatic tuning methods}\label{tuning-lasso}

\noindent\textbf{\textsf{Cross-validation.}} Arguably, $V$-fold Cross-Validation is the most popular technique for tuning the above-mentioned procedures. 
%In the specific case of the relaxed-Lasso and under rather strong assumptions, Meinshausen~\cite{Meinshausenrelaxed} ensures that $V$-fold \CV\  selects a parameter $\lambda$ that asymptotically performs  almost as well as (up to a $\log(n) $ multiplicative term) the oracle~\eref{tuning_optimal}. 
To our knowledge, 
there are no other theoretical results for $V$-fold \CV\ in large dimensional settings.

 In practice, $V$-fold \CV\ seems to give rather good results.
The problem of choosing the best  $V$ has not yet been solved~\cite[Section 10]{ArlotCelisse}, but
it is often reported that a good choice for $V$ is between 5 and 10. Indeed,  the statistical performance does not increase for larger values of $V$, and averaging over 10 splits remains computationally feasible~\cite[Section 7.10]{hastie2009}.
\smallskip

\noindent\textbf{\textsf{LinSelect.}} The procedure \LinSelect\ can be used for selecting among a collection $(\hbeta_{\lambda})_{\lambda\in\Lambda}$ of sparse regressors as follows. For $\J\subset \ac{1,\ldots,p}$, we define $\X_{\J}$ as the matrix $[\X_{ij}]_{i=1,\ldots,n,\ j\in \J}$ obtained by only keeping the columns of $\X$ with index in $\J$. We recall that  the collection $\S$  gives some possible "approximative" supports for the estimators $(\X\hbeta_{\lambda})_{\lambda\in\Lambda}$. For  sparse linear regression, a possible collection $\S$ and measure of complexity $\Delta$ are 
\begin{eqnarray*}
\S&=&\Big\{S=\range(\X_{\J}),\ \J\subset \ac{1,\ldots,p},\ 1\leq|\J|\leq n/(3\log p)\Big\}\\
\textrm{and}\quad\Delta(S)&=&\log\binom{p}{\dim(S)}+\log(\dim(S)).
\end{eqnarray*}
Let us introduce the spaces $\widehat S_{\lambda}=\range\pa{\X_{\supp(\hbeta_{\lambda})}}$ and the subcollection of $\S$
$$\widehat\S=\ac{\widehat S_{\lambda},\ \lambda\in\widehat\Lambda},\quad\textrm{where}\ \widehat\Lambda=\ac{\lambda\in\Lambda\ :\ \widehat S_{\lambda}\in\S}.$$
The following proposition gives a risk bound when selecting $\widehat\lambda$ with \LinSelect\ with   the above choice of $\widehat \S$ and $\Delta$.
\begin{prte}\label{oracle-sparse}
There exists a numerical constant $C>1$ such that  for any minimizer $\widehat{\lambda}$ of the Criterion~\eref{eq:linselect}, we have
\begin{eqnarray}\label{linselect-sparse}
\lefteqn{\mathcal{R}\left[\widehat{\beta}_{\widehat{\lambda}};\beta_0\right]}\nonumber\\
&\leq& 
 C\ \E\cro{\inf_{\lambda\in\Lambda} \ac{\|{\bf X}\hbeta_{\lambda}-{\bf X}\beta_{0}\|_2^2+
 \inf_{S\in\widehat \S}\ac{\|\X\hbeta_{\lambda}-\Pi_{S}\X\hbeta_{\lambda}\|^2_{2}+\dim(S)\log(p)\sigma^2}}} \nonumber\\
&\leq& C\ \E\cro{\inf_{\lambda\in\widehat\Lambda} \ac{\|{\bf X}\hbeta_{\lambda}-{\bf X}\beta_{0}\|_2^2+\|\hbeta_{\lambda}\|_{0}\log(p)\sigma^2}}.
%  \inf_{\lambda\in\Lambda }\left\{\mathcal{R}[\widehat{\beta}_{\lambda};\beta_0]+ \mathbb{E}\left[\text{dim}(S_{\widehat{m}_\lambda})\right]\log(p)\sigma^2\right\}
\end{eqnarray}
\end{prte}
Proposition~\ref{oracle-sparse} is a simple corollary of Proposition~\ref{ORACLE-LINSELECT} in Appendix~\ref{appendix-linselect}.   The first bound involves three terms: the loss of the estimator $\hbeta_{\lambda}$, an approximation loss, and a variance term. Hence, $\LinSelect$ chooses an estimator $\hbeta_{\lambda}$ that achieves a trade-off between the loss of $\hbeta_{\lambda}$ and the closeness of ${\bf X}\hbeta_{\lambda}$ to some small dimensional subspace $S$. 
The bound~\eref{linselect-sparse} cannot be formulated in the form~\eref{bel-oracle} due to the random nature of the set $\widehat\Lambda$. Nevertheless, a bound similar to~\eref{borne_optimales_risque} can be deduced from~\eref{linselect-sparse}  when the estimators $\hbeta_{\lambda}$ are least-squares estimators,
%fulfill ${\bf X}\widehat\beta_{\lambda}=\Pi_{\widehat S_{\lambda}}Y$, 
see Corollary~4 in~\cite{linselect}. Furthermore, we note that increasing the size of $\Lambda$ leads
 to a better risk bound for $\hbeta_{\widehat{\lambda}}$. It is then advisable to consider a family of candidate estimators $\{\hbeta_{\lambda},\ \lambda\in\Lambda\}$ as large as possible. The Proposition \ref{oracle-sparse} is valid for any family of estimators $\{\hbeta_{\lambda},\ \lambda\in\Lambda\}$, for the  specific family of Lasso estimators $\{\widehat{\beta}^{L}_{\lambda},\ \lambda>0\}$  we provide a refined bound in  Proposition~\ref{prte_risque_proba_LinSelect_Lasso}, Section~\ref{result-lasso}.

\subsection{Lasso-type estimation under unknown variance}\label{section_tune_lasso}
The Lasso is certainly one of the most popular methods for variable
selection in a high-dimensional setting. Given $\lambda>0$, the Lasso estimator 
$\widehat{\beta}^{L}_{\lambda}$ is defined by  $\widehat{\beta}^{L}_{\lambda}:= \argmin_{\beta\in\mathbb{R}^p}\|Y-{\bf X}\beta\|_2^2+\lambda \|\beta\|_1$.
A sensible choice of $\lambda$ must be homogeneous with the square-root of the  variance $\sigma^2$. As explained above, when the variance $\sigma^2$ is unknown,  one may apply $V$-fold \CV\ or \LinSelect~to select $\lambda$. Some alternative approaches have also  been developed for tuning the Lasso. Their common idea is to modify the $\ell_1$ criterion  so that the tuning parameter becomes pivotal with respect to $\sigma^2$. This means that the method remains valid for any $\sigma>0$ and that the choice of the tuning parameter does not depend on $\sigma$. For the sake of simplicity, we assume throughout this subsection and the next one that the columns of ${\bf X}$ are normalized to one.
\smallskip

\noindent\textbf{\textsf{$\ell_1$-penalized log-likelihood.}} In low-dimensional  regression, it is classical to consider a penalized log-likelihood criterion instead of 
a penalized least-square  criterion to handle the unknown variance. Following this principle,  St\"adler et al.~\cite{stadler} propose to minimize the $\ell_1$-penalized log-likelihood criterion
\begin{eqnarray}\label{critere_vraisemblance
}
 \widehat{\beta}^{LL}_{\lambda}, \widehat{\sigma}^{LL}_{\lambda}:=\argmin_{\beta\in\mathbb{R}^p,\sigma'>0}\left[n\log(\sigma')+\frac{\|Y-{\bf X}\beta\|_2^2}{2\sigma'^2}+ \lambda\frac{\|\beta\|_1}{\sigma'}\right].
\end{eqnarray}
By reparametrizing $(\beta,\sigma)$, St\"adler et al.~\cite{stadler} obtain a convex criterion that can be efficiently minimized. Interestingly, the penalty level $\lambda$ is pivotal with respect to $\sigma$. Under suitable conditions on the design matrix ${\bf X}$,  Sun and Zhang~\cite{sundiscussion} show that  the choice $\lambda=c\sqrt{2\log p}$, with $c>1$ yields optimal risk bounds in the sense of (\ref{borne_optimales_risque}).
\smallskip

\noindent\textsf{\textbf{Square-root Lasso and scaled Lasso}}.
%\begin{eqnarray}
%\label{critere_vraisemblance_huber}
% \widehat{\beta}^{SL}_{\lambda}, \widehat{\sigma}^{SL}_{\lambda}:=\argmin_{\beta\in\mathbb{R}^p,\sigma'>0}\left[n\sigma'+\frac{\|Y-{\bf X}\beta\|_2^2}{2\sigma'}+ \lambda\|\beta\|_1\right]\ .
%\end{eqnarray}
Sun and Zhang~\cite{sun},
following an idea of Antoniadis~\cite{anto2011}, propose to minimize a penalized Huber's  loss~\cite[page 179]{Huber}
\begin{eqnarray}\label{critere_vraisemblance_huber}
 \widehat{\beta}^{SR}_{\lambda},
 \widehat{\sigma}^{SR}_{\lambda}:=\argmin_{\beta\in\mathbb{R}^p,\sigma'>0}\left[
   \frac{n\sigma'}{2}+\frac{\|Y-{\bf X}\beta\|_2^2}{2\sigma'}+
   \lambda\|\beta\|_1\right].
\end{eqnarray}
This convex criterion can be minimized with roughly the same
computational complexity as a Lars-Lasso
path~\cite{lars}. Interestingly, their procedure (called the scaled Lasso in~\cite{sun}) is equivalent to  the square-root Lasso estimator previously introduced by Belloni et al.~\cite{squarerootlasso}. 
The square-root Lasso of Belloni et al.  is obtained  by replacing 
the residual sum of squares in the Lasso criterion by its square-root
\begin{equation}\label{critere_square_root_lasso}
 \widehat{\beta}^{SR}_{\lambda}= \argmin_{\beta\in\mathbb{R}^p}\left[\sqrt{\|Y-{\bf X}\beta\|_2^2}+\frac{\lambda}{\sqrt{n}} \|\beta\|_1\right]\,.
\end{equation}
The equivalence between the two definitions follows from the minimization of the criterion in (\ref{critere_vraisemblance_huber}) with respect to $\sigma'$.
In (\ref{critere_vraisemblance_huber}) and (\ref{critere_square_root_lasso}), the penalty level $\lambda$ is again pivotal with respect to $\sigma$.
Sun and Zhang~\cite{sun} state sharp
oracle inequalities for the estimator $\widehat{\beta}^{SR}_{\lambda}$ with
$\lambda=c\sqrt{2\log(p)}$, with $c>1$ (see
Proposition~\ref{prte_risque_lassos} below). Their empirical results
suggest that the criterion (\ref{critere_square_root_lasso})
provides slightly better results than the $\ell^1$-penalized
log-likelihood. In the sequel, we shall refer to $\widehat{\beta}^{SR}_{\lambda}$ as the square-root Lasso estimator.
\smallskip

% 
%\noindent\textsf{\textbf{Square-root Lasso.}} Belloni et al.~\cite{squarerootlasso} propose to replace the residual sum of squares in the Lasso criterion by its square-root
%\begin{equation}\label{critereSR}
% \widehat{\beta}^{SR}= \argmin_{\beta\in\mathbb{R}^p}\left[\sqrt{\|Y-{\bf X}\beta\|_2^2}+\frac{\lambda}{\sqrt{n}} \|\beta\|_1\right]\,.
%\end{equation}
%Interestingly, the penalty level $\lambda$ is again pivotal. For $\lambda= c\sqrt{2\log(p)}$ with $c>1$, the square-root Lasso estimator also achieves optimal risk bounds under suitable assumptions on the design ${\bf X}$ (see  Proposition~\ref{prte_risque_lassos} below).\smallskip

\noindent\textsf{\textbf{Bayesian Lasso.}}
 The Bayesian paradigm allows to put prior distributions on the variance $\sigma^2$ and the tuning parameter $\lambda$, as in the Bayesian Lasso~\cite{bayesianlasso}. Bayesian procedures straightforwardly handle the case of unknown variance, but no frequentist analysis of these procedures are so far available.

\subsection{Risk bounds for square-root Lasso and Lasso-\LinSelect}\label{result-lasso}
Let us state a  bound on the prediction error for the square-root Lasso (also called scaled Lasso). For the sake of conciseness, we only present a simplified version of Theorem~1 in~\cite{sun}.
Consider some number $\xi>0$ and some subset $T\subset\{1,\ldots,p\}$. The compatibility constant $\kappa[\xi,T]$ is defined by
\begin{equation*}
 \kappa[\xi,T]= \min_{u\in \mathcal{C}(\xi,T)} \left\{\frac{|T|^{1/2}\|{\bf X}u\|_2}{\|u_T\|_1}\right\},\ \ \textrm{where}\ 
  \mathcal{C}(\xi,T)=\{u:\ \|u_{T^c}\|_1< \xi \|u_T\|_1\}. 
\end{equation*}
\begin{prte}\label{prte_risque_lassos}
There exist positive numerical constants $C_1$, $C_2$, and $C_3$ such that the following holds.
Let us consider the square-root Lasso with
the tuning parameter  $\lambda=2\sqrt{2\log(p)}$. If we assume that
\begin{enumerate}
 \item $p\geq C_1$ 
\item $ \|\beta_0\|_0\leq C_2 \,\kappa^2[4,\mathrm{supp}(\beta_0)]\,\frac{n}{\log(p)}$,
\end{enumerate}
then, with high probability, 
\begin{eqnarray*}
 \|{\bf X}(\widehat{\beta}^{SR}-\beta_{0})\|_2^2&\leq& \inf_{\beta\neq 0}\left\{\|{\bf X}(\beta_0-\beta)\|_2^2+ C_3\frac{\|\beta\|_0\log(p)}{\kappa^2[4,\mathrm{supp}(\beta)]}\sigma^2\right\}\ .  
\end{eqnarray*}
\end{prte}
This bound is  comparable to the general  objective (\ref{bel-oracle})  stated in Section~\ref{section_limites_theorique_object}. Interestingly, the constant before the bias term  $\|{\bf X}(\beta_0-\beta)\|_2^2$ equals one. If $\|\beta_0\|_0=k$, the square-root Lasso achieves the minimax loss $k\log(p)\sigma^2$ as long as $k\log(p)/n$ is small  and $\kappa[4,\mathrm{supp}(\beta_0)]$ is away from zero. This last condition ensures that the design ${\bf X}$ is not too far from orthogonality on the cone $\mathcal{C}(4,\mathrm{supp}(\beta_0))$.
State of the art results for the classical Lasso with known variance~\cite{bickeltsy08, Kolt11,geer_condition} all involve this condition. \smallskip

We next state a risk bound for the Lasso-\LinSelect\ procedure.
For $\J\subset \ac{1,\ldots,p}$, we define 
$\phi_{\J}$ as the largest eigenvalue of $X_{\J}^TX_{\J}$. The following proposition involves the restricted eigenvalue
$\phi_{*}=\max\ac{\phi_{\J}:\textrm{Card}(\J)\leq n/(3\log p)}.$
%When the regularization parameter $\lambda\in\R^+$ of the Lasso is selected according to the \LinSelect\ criterion~(\ref{eq:linselect}), we have the following upper-bound for the resulting estimator.
%%%%%%%%%%%
\begin{prte}\label{prte_risque_proba_LinSelect_Lasso}
%\begin{eqnarray*}
%\S&=&\Big\{S=\range(\X_{\J}),\ \J\subset \ac{1,\ldots,p},\ |\J|\leq n/(3\log(p)\Big\}\\
%\textrm{and}\quad\Delta(S)&=&\log\binom{p}{\dim(S)}+\log(\dim(S)).
%\end{eqnarray*}
There exist positive numerical constants $C$, $C_1$, $C_2$, and $C_3$ such that the following holds.
Take $\Lambda=\mathbb{R}^+$ and assume that
$$\|\beta_{0}\|_{0}\leq C\,{\kappa^2[5,\supp(\beta_{0})]\over\phi_{*}}\times {n\over \log(p)}.$$
Then, with probability at least $1-C_1p^{-C_2}$, the Lasso estimator  $\hbeta_{\widehat\lambda}^L$ selected according to the \LinSelect\ procedure described in Section~\ref{tuning-lasso} fulfills
\begin{equation}\label{eq:risque_proba_Linselect_lasso}
%\norm{{\bf X}(\beta_{0}-\widehat \beta_{\widehat \lambda}^L)}^{2}_{2} \le C_{3}\inf_{\beta\neq 0}\ac{\|\X(\beta_{0}-\beta)\|^2_{2}+\frac{\phi_{*}}{ \kappa^2[5,\supp(\beta)]}\,\|\beta\|_{0}\log(p) \sigma^2}.
\norm{{\bf X}(\beta_{0}-\widehat \beta_{\widehat \lambda}^L)}^{2}_{2} \le C_{3}\inf_{\beta\neq 0}\left\{\|\X(\beta_{0}-\beta)\|^2_{2}+\frac{\phi_{*}\|\beta\|_0\log(p)}{ \kappa^2[5,\supp(\beta)]}\, \sigma^2 \right\}\ . 
\end{equation}
\end{prte}
The bound~(\ref{eq:risque_proba_Linselect_lasso}) is similar to the bound stated above for the square-root Lasso, the most notable differences being the constant larger than 1 in front of the bias term  and the quantity $\phi_*$ in front of the variance term. We refer to the Appendix~\ref{proof:oracle-lasso-glasso} for a proof of Proposition~\ref{prte_risque_proba_LinSelect_Lasso}.

\subsection{Support estimation and inverse problem}
Until now, we only discussed estimation methods that perform well in prediction. Little is known
when the objective is to infer  $\beta_0$  or its support under unknown variance. 
\smallskip

\noindent\textsf{\textbf{Inverse problem.}} The square-root Lasso~\cite{sun,squarerootlasso} is proved to achieve near optimal risk bound for the inverse problems under suitable assumptions on the design $\X$. 
\smallskip

\noindent\textsf{\textbf{Support estimation.}} Up to our knowledge, there are no non-asymptotic results on support estimation for the aforementioned procedures in the unknown variance setting. Nevertheless, some related results and heuristics have been developed for the cross-validation scheme. If the tuning parameter $\lambda$ is
chosen to minimize the prediction error (that is take $\lambda=\lambda^*$ as defined in (\ref{tuning_optimal})), the Lasso is not consistent
for support estimation  (see~\cite{LLW06,MB06} for results in a random design setting). One idea to overcome this problem, is to choose
the parameter $\lambda$ that minimizes the risk of the so-called Gauss-Lasso estimator $\widehat{\beta}^{GL}_{\lambda}$ which is the least square
estimator over  the support of the Lasso estimator $\widehat{\beta}^{L}_{\lambda}$
\begin{equation}\label{definition_gauss_lasso}
 \widehat{\beta}^{GL}_{\lambda}:= \argmin_{\beta\in\mathbb{R}^p: \mathrm{supp}(\beta)\subset \mathrm{supp}(\widehat{\beta}^{L}_{\lambda})}\|Y-{\bf X}\beta\|_2^2\ .
\end{equation}
When the objective is support estimation, some numerical simulations~\cite{tsyrig10} suggest that it may be more advisable not to apply the selection schemes based on prediction risk (such as $V$-fold \CV\  or \LinSelect) to 
 the Lasso estimators but rather to  the  Gauss-Lasso estimators.
Similar remarks also apply for the Dantzig Selector~\cite{candes07}.

\subsection{Numerical Experiments}\label{NumExp.st}
We present two numerical experiments to illustrate the behavior of some of the above mentioned procedures for high-dimensional sparse linear regression. 
 The first one concerns the problem
of  tuning the parameter $\lambda$ of the Lasso algorithm for
estimating $\X \beta_{0}$. The procedures will be compared on the basis
of the prediction risk.
The second one concerns the problem of support estimation with
Lasso-type estimators. We will focus on the  false discovery rates (FDR) and the proportion of true discoveries (Power).
\smallskip

\noindent\textsf{\textbf{Simulation design.}}
The simulation design is the same as the one described
in Sections 6.1, and 8.2 of~\cite{linselect}, except that we restrict
to the case $n=p=100$. Therefore, 165 examples are  simulated. They are
inspired by examples
found in~\cite{tiblasso, zou05, zou_adaptive, 
  2008_Huang} and cover a large variety of situations.
The
simulation were carried out with {\tt R  (www.r-project.org)}, using
the library {\tt elasticnet}.
\smallskip

\noindent\textsf{\textbf{Experiment 1 : tuning the Lasso for prediction.}}\\
In the first experiment, we compare 10-fold \CV~\cite{Geisser75},
  \LinSelect~\cite{linselect} and the square-root Lasso~\cite{squarerootlasso,sun} (also called scaled Lasso)  for tuning the Lasso. Concerning the square-root Lasso, we set  $\lambda = 2\sqrt{2 \log(p)}$ (as suggested in \cite{sun}) and we compute the estimator using the algorithm described in Sun and Zhang~\cite{sun}.

For each
  tuning  procedure $\ell \in$ $\{\mbox{10-fold \CV},
  \mbox{\LinSelect}, \mbox{ square-root Lasso} \}$, we focus on the prediction risk ${\mathcal
  R}\Big[\widehat{\beta}^L_{\hat{\lambda}_{\ell}} ; \beta_{0}\Big]$ of the selected Lasso estimator $\widehat{\beta}^L_{\hat{\lambda}_{\ell}}$.\smallskip

For each simulated example $e =1, \ldots, 165$, we estimate on the basis of 400 runs
\begin{itemize}
\item the risk of the oracle~\eref{tuning_optimal} : ${\mathcal R}_{e} = {\mathcal
  R}\left[\widehat{\beta}_{\lambda^*} ; \beta_{0}\right]$,
\item the 
risk when selecting $\lambda$ with procedure $\ell$ : 
${\mathcal R}_{\ell, e} = {\mathcal
  R}\left[\widehat{\beta}_{\hat{\lambda}_{\ell}} ; \beta_{0}\right]$.
\end{itemize}

The comparison between the procedures is based on the comparison of the
means, standard deviations and quantiles of the risk ratios ${\mathcal
  R}_{\ell, e}/{\mathcal R}_{e}$ computed over all the simulated
examples $e=1, \ldots, 165$. The results are displayed in
Table~\ref{RiskLasso.tb}. 

\begin{table}[hptb]
\begin{tabular}{r|rr|rrrrr}
&&&\multicolumn{5}{c}{quantiles} \\ 
procedure&        mean & std-err & $0\%$ & $50\%$ & $75\%$ & $90\%$ & $95\%$ \\ \hline
Lasso $10$-fold \CV & 1.13 & 0.08 & 1.03 & 1.11 & 1.15  & 1.19 &1.24 \\
Lasso \LinSelect  &      1.19 & 0.48 & 0.97 & 1.03 & 1.06 & 1.19 &2.52\\
Square-root Lasso &    5.15 & 6.74 & 1.32 & 2.61 & 3.37 & 11.2 & 17 \\
\end{tabular}
\caption{\label{RiskLasso.tb} For each procedure $\ell$, mean,
  standard-error and quantiles of the ratios $\{{\mathcal R}_{\ell, e}/{\mathcal R}_{e},
e=1, \ldots, 165\}$.
}
\end{table}

For  10-fold \CV\ and \LinSelect, the risk ratios  are close to one.  For $90\%$ of the examples, the risk of the Lasso-\LinSelect\  is smaller than the risk of the Lasso-\CV,  but there are a few examples where the risk of the Lasso-\LinSelect\  is significantly larger than the risk of the Lasso-\CV.
 For the  square-root Lasso procedure,  the risk ratios are clearly larger
than for the two others.
An inspection of the results reveals that the square-root Lasso
 selects estimators with supports  of small size.  
 This feature can be interpreted as follows. Due to the bias of the Lasso-estimator, the residual variance tends to over-estimate the variance, leading the square-root Lasso to select a Lasso estimator $\hbeta_{\lambda}^L$ with large $\lambda$. 
  Consequently the risk is high.

%{\color{magenta} j'ai enleve la table 2 et les discussions qui vont avec}
%  We illustrate this phenomenon in  Table~\ref{RiskLasso.tb2}, where we consider apart \emph{easy examples} with  ratio ${\mathcal R}[\widehat{\beta}_{\lambda^*} ; \beta_{0}]/n \sigma^{2}$ smaller than 0.1 and \emph{difficult examples} with ratio ${\mathcal R}[\widehat{\beta}_{\lambda^*} ; \beta_{0}]/n \sigma^{2}$ greater than 0.8.  The bias of the Lasso estimator is small for easy examples, whereas it is large for difficult examples. The difference of behavior between the scaled Lasso and the two other procedures is significant.

%\begin{table}[hptb]
%\begin{tabular}{r|cc|cc}
%                    &   \multicolumn{2}{c}{\emph{easy examples}} &   \multicolumn{2}{c}{\emph{difficult examples}}\\ 
%                    & mean risk ratio & $\|\widehat{\beta}^{L}_{\hat{\lambda}_{\ell}}\|_{0}$  & mean risk ratio & $\|\widehat{\beta}^{L}_{\hat{\lambda}_{\ell}}\|_{0}$   \\ \hline
%Oracle             &   1 & (11)    &  1 &(43) \\
%Lasso $V$-fold \CV &   1.35& (13)  &  1.12& (45)\\
%Lasso \LinSelect   &   1.06 &(11)  &  1.71& (40)\\
%Scaled Lasso       &    1.89& (8)  &  15& (3.5)\\
%\end{tabular}
%\caption{\label{RiskLasso.tb2} 
%Mean risk ratio ${\mathcal R}_{\ell, e}/{\mathcal R}_{e}$ for each  selection procedure $\ell$ on \emph{easy} and \emph{difficult} examples. In parentheses :  $\|\widehat{\beta}^{L}_{\hat{\lambda}_{\ell}}\|_{0}$. The first line gives the same values for the oracle $\widehat{\beta}^{L}_{\lambda^{\star}}$}
%\end{table}
\smallskip

\noindent\textsf{\textbf{Experiment 2 : variable selection with Gauss-Lasso and square-root Lasso.}}\\
We consider now the problem of support estimation, sometimes referred as the problem of variable selection. 
We implement three procedures. The Gauss-Lasso procedure tuned by either 10-fold \CV\ or \LinSelect\ and the square-root Lasso. The support of $\beta_{0}$ is estimated by the support of the selected estimator.

For each simulated example, the FDR and the Power are estimated on the
basis of 400 runs. The results are given on 
%Figure~\ref{FDRPower.fg}.
Table~\ref{FDRPower.tb}.
%\begin{figure}[htbp]
%\begin{center}
%{\includegraphics[width=8cm,angle=270]{BPGlassoSZFDRPower}}
%\end{center}
%\caption{\label{FDRPower.fg} For each procedure boxplots of FDR and
%  Power values}
%\end{figure}

\begin{table}[hptb]
\begin{tabular}{r|rr|rrrrr}
\multicolumn{8}{c}{False Discovery rate} \\ 
&&&\multicolumn{5}{c}{quantiles} \\ 
procedure&        mean & std-err & $0\%$ & $25\%$& $50\%$ & $75\%$ & $90\%$  \\ \hline
Gauss-Lasso $10$-fold \CV & 0.28 & 0.26 & 0 & 0.08 & 0.22 & 0.35 & 0.74 \\
Gauss-Lasso \LinSelect  &  0.12 & 0.25 & 0 & 0.002 &0.02 & 0.13 & 0.33 \\
Square-root Lasso &       0.13 & 0.26 & 0 & 0.009 &0.023   & 0.07 & 0.32  \\
\hline \\ 
\multicolumn{8}{c}{Power} \\ 
&&&\multicolumn{5}{c}{quantiles} \\ 
procedure&        mean & std-err & $0\%$ & $25\%$ & $50\%$ & $75\%$ & $90\%$ \\ \hline
Gauss-Lasso $10$-fold \CV & 0.67 & 0.18 & 0.4 & 0.52 & 0.65 & 0.71 & 1 \\
Gauss-Lasso \LinSelect  &  0.56 & 0.33 & 0.002   & 0.23 & 0.56 & 0.93 & 1\\
Square-root Lasso &       0.59 & 0.28 & 0.013   & 0.41 & 0.57 & 0.80 & 1  \\
\end{tabular}
\caption{\label{FDRPower.tb} For each procedure $\ell$, mean,
  standard-error and quantiles of FDR and Power values.
}
\end{table}

It appears that the Gauss-Lasso \CV ~procedure gives greater values of the
FDR than the two others. The Gauss-Lasso \LinSelect ~and the
square-root Lasso behave similarly for the FDR, but the values of the power
are more variable for the \LinSelect ~procedure.\smallskip

\noindent\textsf{\textbf{Computation time.}}\\* 
Let us conclude this numerical section with the comparison of the
computation times between the methods. For all methods the computation
time depends on the maximum number of steps in the lasso algorithm and
for the LinSelect method, it depends on the cardinality of $\S$ or equivalently
on the maximum  number of non-zero components of
$\widehat{\beta}$.  The results are shown at
Table~\ref{CompTime.tb}. The square-root Lasso  is the less time consuming
method, closely  followed by the Lasso \LinSelect\ method. The
$V$-fold \CV\, carried out with the function {\tt cv.enet} of the {\tt
  R} package {\tt elasticnet}, pays the price of several calls to the
lasso algorithm.

\begin{table}[hptb]
%\begin{tabular}{r|rrr}
% &Lasso $V$-fold \CV & Lasso \LinSelect   &Scaled Lasso 
%\\ \hline
%$n=p=100$,          & \\ 
%{\tt max.steps}=100  &  4 s & 0.21 s & 0.18 s\\ 
% $k_{\max}=21$       & \\ \hline
%$n=100$, $p=500$ &\\ 
%{\tt max.steps}=100 & 4.8 s&  0.43 s   &0.4 s\\ 
% $k_{\max}=16$ & \\ \hline
%$n=p=500$,          & \\ 
%{\tt max.steps}=500  & 300  s & 11 s & 6.3 s \\ 
% $k_{\max}=80$       & \\ \hline
%\end{tabular}
\begin{tabular}{rrrr|rrr}
$n$ & $p$ &{\tt max.steps} & $k_{\max}$ &Lasso $10$-fold \CV & Lasso \LinSelect   &Square-root Lasso 
\\ \hline
100&100  & 100 &21 & 4 s & 0.21 s & 0.18 s\\ \hline
100 &500 & 100 &16 & 4.8 s&  0.43 s   &0.4 s\\ \hline
500&500 & 500 &80 & 300  s & 11 s & 6.3 s \\ 
\end{tabular}
\caption{\label{CompTime.tb} For each procedure computation time for
  different values of $n$ and $p$. The maximum number of
  steps in the lasso algorithm, is taken as ${\tt
    max.steps}=\min\left\{n,p\right\}$. For the \LinSelect\ procedure, the maximum  number of
  non-zero components of $\widehat{\beta}$, denoted  $k_{\max}$ is
  taken as $k_{\max} = \min\left\{p, n/\log(p)\right\}$. 
}
\end{table}

%\begin{table}[hptb]
%\begin{tabular}{r|rrr}
%procedure&               Number & mean FDR & mean POWER \\ \hline
%Gauss Lasso $V$-fold \CV  &   34   & 0.013    &   0.67 \\
%Gauss-Lasso \LinSelect  &  108   & 0.012    &   0.57 \\
%Scaled Lasso           &  108   &  0.017   &   0.62  \\
%\end{tabular}
%\caption{\label{FDRPower.tb} For each procedure :  number of
%  simulated examples such that the FDR is smaller than $5\%$, mean of
%  the FDR and of 
%  the power for those examples.
%}
%\end{table}

%%% Local Variables: 
%%% mode: latex
%%% TeX-master: "toto"
%%% TeX-master: "NumExp"
%%% End: 

\section{Group-sparsity}\label{subsection_groupe}
In the previous section, we have made no prior assumptions on the form of $\beta_0$. In some applications, there are some known structures between the covariates. As an example, we treat the now classical case of group sparsity.
The covariates are assumed to be clustered into $M$ groups and when the coefficient $\beta_{0,i}$ corresponding to the covariate ${\bf X}_i$ is non-zero then it is likely that all the coefficients $\beta_{0,j}$ with variables ${\bf X}_j$ in the same group as ${\bf X}_i$ are non-zero. We refer to the introduction of  \cite{groupbach} for practical examples of this so-called group-sparsity assumption. Let $G_1,\ldots, G_M$ form a given partition of $\ac{1,\ldots, p}$. For $\lambda=(\lambda_1,\ldots,\lambda_M)$, the group-Lasso estimator  $\widehat{\beta}_{\lambda}$ is defined as the minimizer of the convex optimization criterion
\begin{equation}\label{critere_groupe}
 \|Y-{\bf X}\beta\|^2_2+\sum_{k=1}^M\lambda_k\|\beta^{G_{k}}\|_2\,,
\end{equation}
where  $\beta^{G_{k}}=(\beta_j)_{j\in G_k}$. 
The Criterion~\eref{critere_groupe} promotes solutions where all the coordinates of $\beta^{G_{k}}$ are either zero or non-zero,
 leading to group selection~\cite{groupyuan}. Under some assumptions on ${\bf X}$, Huang and Zhang~\cite{HuangZhang2010} or Lounici {\it et al.}~\cite{tsybakovgroup} provide a suitable choice of $\lambda=(\lambda_1,\ldots,\lambda_M)$ that leads to near optimal prediction bounds. As expected, this choice of  $\lambda=(\lambda_1,\ldots, \lambda_M)$ is proportional to $\sigma$.

As for the Lasso, $V$-fold \CV\ is widely used in practice to tune the penalty parameter $\lambda=(\lambda_1,\ldots,\lambda_M)$. To our knowledge, there is 
not yet any extension of the procedures described in Section \ref{section_tune_lasso} to  the group Lasso. An alternative to cross-validation is to use \LinSelect.
\smallskip

\noindent\textsf{\textbf{Tuning the group-Lasso with \LinSelect.}}
For any $\K\subset \ac{1,\ldots,M}$, we define the submatrix $\X_{(\K)}$ of $\X$ by only keeping the columns of $\X$ with index in $\bigcup_{k\in\K}G_{k}$. We also write $\X_{G_{k}}$ for the  submatrix of $\X$ built from the columns  with index in $G_{k}$. The collection $\S$ and the function $\Delta$ are given by
$$\S=\bigg\{\range(\X_{(\K)}) :\  1\leq |\K|\leq n/(3\log(M))\text{ and } \sum_{k\in \K}|G_{k}|\leq n/2-1\bigg\}$$
and $\Delta\big(\range(\X_{(\K)})\big)=\log\left[|\K|\binom{|\K|}{M}\right]$.  For a given $\Lambda\subset\R_{+}^M$, similarly to Section~\ref{tuning-lasso}, we define $\widehat\K_{\lambda}=\ac{k :\ \|\hbeta^{G_{k}}_{\lambda}\|_2\neq 0}$ and 
$$\widehat\S=\ac{\range(\X_{(\widehat\K_{\lambda})}),\ \lambda\in\widehat\Lambda},\quad \textrm{with}\ \ \widehat\Lambda=\ac{\lambda\in\Lambda,\ \range(\X_{(\widehat\K_{\lambda})})\in\S}.$$ 
Proposition~\ref{ORACLE-LINSELECT} in Appendix~\ref{appendix-linselect} ensures that
 we have for some constant $C>1$
\begin{equation*}%\label{linselect_grouplasso}
 \mathcal{R}\left[\widehat{\beta}_{\widehat{\lambda}};\beta_0\right]\leq 
 C\ \E\cro{\inf_{\lambda\in\widehat\Lambda} \ac{\|\X\hbeta_{\lambda}-\X\beta_{0}\|_2^2+\pa{\|\hbeta_{\lambda}\|_{0}\,\vee\, |\widehat\K_{\lambda}|\log(M)}\sigma^2}}.
\end{equation*}
In the following, we provide a more explicit bound. For simplicity, we restrict to the specific case where each  group $G_k$ has the same cardinality $T$. For $\K\subset \ac{1,\ldots,M}$, we define 
$\phi_{(\K)}$ as the largest eigenvalue of $\X_{(\K)}^T\X_{(\K)}$
and we set
\begin{equation}\label{eq:phi-glasso}
\phi_{*}=\max\ac{\phi_{(\mathcal{K})} :\quad 1\leq|\mathcal{K}|\leq \frac{n-2}{2T\vee3\log(M)}}.
\end{equation}
We assume that all the columns of $\X$ are normalized to 1 and
following Lounici {\it et al.}~\cite{tsybakovgroup}, we introduce for $1\leq s\leq M$
\begin{equation}\label{eq:kappa-glasso}
\kappa_{G}[\xi,s]=\min_{1\leq |\K|\leq s}\ \min_{u\in\Gamma(\xi,\K)}{\|\X u\|_{2}\over \|u_{(\K)}\|_{2}} 
\end{equation}
where $\Gamma(\xi,\K)$ is the cone of vectors $u\in \R^M\setminus\ac{0}$ such that $\sum_{k\in\K^c}\lambda_{k}\|u^{G_{k}}\|_{2}\leq \xi\sum_{k\in\K}\lambda_{k}\|u^{G_{k}}\|_{2}$.
 In the sequel, $\mathcal{K}_0$ stands for the set of groups containing non-zero components of $\beta_0$.
%%%%%%%%%%
\begin{prte}\label{prop_risque_glasso}
There exist positive numerical constants $C$, $C_1$, $C_2$, and $C_3$ such that the following holds.
Assume that $\Lambda$ contains $\bigcup_{\lambda\in\mathbb{R}_+}\ac{(\lambda,\ldots,\lambda)}$, that $T\leq (n-2)/4$ and that 
\begin{eqnarray*}
1\leq|\mathcal{K}_{0}|\leq C\,\frac{\kappa^2_{G}[3,|\mathcal{K}_0|]}{\phi_*}\times \frac{n-2}{\log(M)\vee T}\,.
\end{eqnarray*}
Then, with probability larger than $1-C_1M^{-C_2}$, we have
\begin{equation*}
 \|{\bf X}\hbeta_{\widehat{\lambda}}-{\bf X}\beta_{0}\|_2^2
\leq C_{3}\,\frac{\phi_*}{\kappa^2_G[3,|\mathcal{K}_0|]}\ |\mathcal{K}_{0}|\ \left(T\vee \log(M)\right)\,.
\end{equation*}
\end{prte}
%%%%%%%%%%%
This proposition provides a bound comparable to the bounds  of Lounici {\it et al.}~\cite{tsybakovgroup}, without requiring the knowledge of the variance. Its proof can be found in Appendix~\ref{proof:oracle-lasso-glasso}.

\section{Variation-sparsity}\label{section_rupture}
We focus in this section on the  \emph{variation-sparse} regression. 
We recall that  the vector $\beta^V\in\R^{p-1}$ of the variations  of $\beta$ has for coordinates $\beta_{j}^{V}=\beta_{j+1}-\beta_{j}$
%, $j=1,\ldots,p-1$  
and that the variation-sparse setting corresponds to the setting where the  vector of variations $\beta_{0}^V$ is coordinate-sparse. In the following, we restrict to the case where  $n=p$ and  $\X$ is the identity matrix. In this case, the problem of variation-sparse regression coincides with the problem of segmentation of the mean of the vector $Y=\beta_{0}+\eps$.

For any subset $\cI\subset\ac{1,\ldots,n-1}$, we define $S_{\cI}=\big\{\beta\in\R^n\,:\ \supp(\beta^V)\subset\cI\big\}$ and
$\hbeta_{\cI}=\Pi_{S_{\cI}}Y$.  For any integer $q\in\ac{0,\ldots,n-1}$, we define also the "best" subset of size $q$ by
$$\widehat\cI_{q}=\argmin_{|\cI|=q}\|Y-\hbeta_{\cI}\|_2^2.$$
Though the number of subsets $\cI\subset\ac{1,\ldots,n-1}$ of cardinality $q$  is of order
$n^{q+1}$, this minimization can be performed using dynamic programming with a
complexity of order $n^{2}$~\cite{Guthery}.  To select $\widehat\cI=\widehat\cI_{\hat q}$ with  $\hat{q}$ in $\ac{0,\ldots,n-1}$, any of the generic selection schemes of Section~\ref{generic_schemes.st} can be applied. Below, we instantiate these schemes and present some alternatives.

\subsection{Penalized empirical loss}\label{histo-empirique}
When the variance $\sigma^2$ is known,
penalized log-likelihood model selection amounts to select  
a subset  $\widehat\cI$ which minimizes a criterion  of the form
$\| Y - \hbeta_{\cI} \|_{2}^{2} + \pen(  
\mathrm{Card}(\cI))$. This is equivalent to select $\widehat\cI=\widehat\cI_{\hat q}$ with  $\hat{q}$  minimizing
\begin{equation}\label{critere-rupture}
\crit(q)= \| Y - \hbeta_{\widehat{\cI}_{q}} \|_{2}^{2} + \pen(q).
\end{equation}

 Following the work of Birg\'e and Massart~\cite{BM01}, Lebarbier~\cite{2005_Lebarbier_SignProc}  considers the penalty 
$$ \pen(q) = (q+1)  \left( c_{1} \log(n/(q+1)) + c_{2}\right)\,\sigma^2$$
and determines the constants $c_{1}=2, c_{2}=5$ by extensive numerical
experiments (see also Comte and Rozenholc~\cite{CR04} for a similar approach in a more general setting). With this choice of the penalty, the procedure satisfies a bound of the form
\begin{equation}\label{oracle-rupture}
\mathcal{R}\cro{\widehat\beta_{\widehat \cI},\beta_{0}}\ \leq\ C\,\inf_{\cI\subset \ac{1,\ldots,n-1}}\ac{\|\hbeta_{\cI}-\beta_{0}\|_2^2+(1+ |\cI|)\log(n/(1+ |\cI|))\,\sigma^2}.
\end{equation}

When $\sigma^{2}$ is unknown, several approaches have been 
proposed.
\smallskip

\noindent\textbf{\textsf{Plug-in estimator.}}
The idea is to replace  $\sigma^{2}$ in $\pen(q)$ by an estimator of the variance such as  $\widehat{\sigma}^{2}  = \sum_{i=1}^{n/2} (Y_{2i} - Y_{2i-1})^{2}/n$,
or one of the estimators proposed by Hall and al.~\cite{HKT90}. No theoretical
results are proved in a non-asymptotic framework.
\smallskip

\noindent\textbf{\textsf{Estimating the variance by the residual least-squares.}}
 Baraud et al.~\cite{BGH09} Section 5.4.2 propose to select $q$ by minimizing a penalized log-likelihood criterion. This criterion can be written in the form 
 $\crit(q)= \| Y - \hbeta_{\widehat{\cI}_{q}} \|_{2}^{2}
(1 + K \pen(q))$, 
with $K>1$ and  the penalty $\pen(q)$  solving
\begin{equation*}
\E \left[\left( U - {\pen(q)} V\right)_{+}\right] =
 \frac{1}{(q+1)\binom{n-1}{q}}\ ,
\end{equation*}
 where $(.)_{+}=\max(.,0)$, and
$U$, $V$ are two independent $\chi^{2}$ variables with
respectively $q+2$ and $n-q-2$ degrees of freedom. The resulting estimator $\hbeta_{\widehat\cI}$, with $\widehat\cI=\widehat\cI_{\hat q}$, 
satisfies a non asymptotic risk bound similar to~\eref{oracle-rupture} for all $K>1$.  The choice $K=1.1$ is suggested for the practice.
\smallskip

\noindent\textbf{\textsf{Slope heuristic.}}
Lebarbier~\cite{2005_Lebarbier_SignProc} implements the slope heuristic
introduced by Birg\'e and Massart~\cite{massart_pente} for handling the unknown variance $\sigma^{2}$. The method consists in calibrating the
penalty directly, without estimating $\widehat{\sigma}^{2}$.
It is based on the following principle. First, there exists a so-called \emph{minimal} penalty $\pen_{\min}(q)$ such that choosing $\pen(q)=K\pen_{\min}(q)$ in~\eref{critere-rupture} with $K<1$ can lead to a strong  overfit, whereas for $K>1$ the bound~\eref{oracle-rupture} is met. 
Second, it  can be shown that there exists
a  \emph{dimension jump} around the minimal penalty, allowing  to
estimate $\pen_{\min}(q)$ from the data. The slope heuristic then proposes to select $q$ by minimizing the criterion
 $\crit(q)= \| Y - \hbeta_{\widehat{\cI}_{q}} \|_{2}^{2}
 + 2\,\widehat{\pen}_{\min}(q)$.
Arlot and
Massart~\cite{2010_ArlotMassart_JMLR} provide a non asymptotic risk bound for this procedure. 
Their results are proved in a general regression model with
heteroscedatic and non Gaussian errors, but with a constraint on the number of models per dimension which is not met for the family of  models $(S_{\cI})_{\cI\subset\ac{1,\ldots,n-1}}$. Nevertheless, the authors indicate how to
generalize their results for the problem of signal segmentation.

Finally, for practical issues, different procedures for estimating the
minimal penalty  are compared and
implemented in Baudry et al.~\cite{BMM10}.

\subsection{CV procedure}

In a recent paper, Arlot and
C\'elisse~\cite{2010_ArlotCelisse_StaComp} consider the problem of
 signal segmentation using cross-validation. Their
results apply in the  heteroscedastic case. They
consider several \CV-methods, the leave-one-out, 
leave-$p$-out and $V$-fold \CV\ for estimating the quadratic loss. They
propose two cross-validation schemes. The first one, denoted \emph{Procedure
  5}, aims to estimate directly $\E\left[\|\beta_{0} - \beta_{\widehat{I}_{q}}\|_2^{2}\right]$, while the second one, denoted  \emph{Procedure 6}, relies on two steps where the cross-validation is used first for choosing the
best partition of dimension $q$, then the best dimension $q$. They
show that the leave-$p$-out \CV\ method can be
implemented with a complexity of order $n^{2}$, and they give a control of 
the expected \CV\ risk. The use of CV leads to some restrictions on
the subsets $\cI$ that compete for estimating $\beta_{0}$. This problem
is discussed in~\cite{2010_ArlotCelisse_StaComp}, Section 3 of the
supplemental material.

\subsection{Alternative for very high-dimensional settings}

When $n$ is very large, the dynamic programming optimization can become computationally  too intensive.  An attractive alternative is based on the fused Lasso proposed by Tibshirani et
al.~\cite{2005_TSR_JRSSB}. The estimator $\hbeta_{\lambda}^{TV}$ is defined by minimizing the convex criterion
$$\|Y-\beta\|^2_{2}+\lambda \sum_{j=1}^{n-1} |\beta_{j+1}-\beta_{j}|,$$
where the total-variation norm $\sum_j|\beta_{j+1}-\beta_{j}|$ promotes solutions which are variation-sparse.
The family $(\hbeta^{TV}_{\lambda})_{\lambda\geq0}$ can be computed
very efficiently with the LARS-algorithm, see Vert and
Bleakley~\cite{VertNIPS2010}. A sensible choice of the parameter $\lambda$ must be proportional to $\sigma$. When the variance $\sigma^2$ is unknown,  the parameter $\lambda$ 
 can be selected  either by $V$-fold \CV\ 
or by \LinSelect\ (see Section 5.1 in~\cite{linselect} for details). 
%An alternative is to estimate directly $\beta^{V}_{0}$ with  the scaled Lasso~\cite{sun}. 

\section{Extensions}\label{section_extension}

\subsection{Gaussian design and graphical models} Assume that the design ${\bf X}$ is now random and that the $n$ rows ${\bf X}^{(i)}$ are independent observations of a Gaussian vector with mean $0_p$ and unknown covariance matrix $\Sigma$. This setting is mainly motivated by applications in compressed sensing~\cite{donoho_compressed} and  in Gaussian graphical modeling. Indeed, Meinshausen and B\"uhlmann~\cite{MB06} have proved that it is possible to estimate the graph of a Gaussian graphical model by studying linear regression with Gaussian design and unknown variance. If we work 
 conditionally  on the observed ${\bf X}$ design, then all the results and methodologies described in this survey still apply. Nevertheless, these prediction results do not really take into account the fact that the design is random. 
 In this setting,  it is more  natural to consider 
 the integrated prediction risk $\E\big[\|\Sigma^{1/2}(\hbeta-\beta_{0})\|_2^2\big]$ rather than the risk~\eref{definition_risque_prediction}.
Some procedures~\cite{Giraud08,verzelen_regression} have been proved to achieve optimal risk bounds with respect to this risk but they are computationally intractable in a high-dimensional setting.
In the context of Gaussian graphical modeling, the 
  procedure \GGMSelect~\cite{ggmselect} is designed to select among any collection of graph estimators and it is proved  to achieve near optimal risk bounds in terms of the integrated prediction risk.

\subsection{Non Gaussian noise} 
A few results do not require that the noise $\varepsilon$ follows a Gaussian distribution. The Lasso-type procedures such as the square-root Lasso~\cite{sun,squarerootlasso}  do not require the normality of the noise and extend to other distributions. In practice, it seems that cross-validation procedures  still work well for other distributions of the noise.

\subsection{Multivariate regression}
Multivariate regression deals with $T$ simultaneous linear regression models  $y_{k}=\X\beta_{k}+\eps_{k}$, $k=1,\ldots,T$. Stacking the $y_{k}$'s in a $n\times T$ matrix $Y$, we obtain the model
$Y=\X B_{0}+E$, where $B_{0}$  is a $p\times T$ matrix with columns given by $\beta_{k}$ and $E$ is a $n\times T$ matrix with i.i.d.\ entries. The classical structural assumptions on $B_{0}$ are either that most rows of $B_{0}$ are identically zero, or the rank of $B_{0}$ is small. The first case is a simple case of group sparsity and can be handled by the group-lasso as in Section~\ref{subsection_groupe}. The second case, first considered by Anderson~\cite{Anderson51} and Izenman~\cite{Izenman75},  is much more non-linear. Writing $\|.\|_{F}$ for the Frobenius (or Hilbert-Schmidt) norm, the problem of selecting among the estimators
$$\widehat B_{r}=\argmin_{B:\textrm{rank}(B)\leq r}\|Y-\X B\|_{F}^2,\quad r\in \ac{1,\ldots,\min(T,\textrm{rank}(\X))}$$
has been investigated recently from a non-asymptotic point of view by Bunea {\it et al.}~\cite{BuneaSheWegkamp11} and Giraud~\cite{Giraud11}. 
The prediction risk of $\widehat B_{r}$ is of order of
$$\E\cro{\|\X \widehat B_{r}-\X B_{0}\|^2_{F}}\ \asymp\ \sum_{k\geq r} s_{k}^2(\X B_{0})+ r\pa{n+\mathrm{rank}(X)}\sigma^2,$$
where $s_{k}(M)$ denotes the $k$-th largest singular value of the matrix $M$. Therefore, a sensible choice of $r$ depends on $\sigma^2$. The first selection criterion introduced by  Bunea {\it et al.}~\cite{BuneaSheWegkamp11} requires the knowledge of  the variance $\sigma^2$. 
To handle the case of unknown variance,  Bunea {\it et al.}~\cite{BuneaSheWegkamp11} propose to plug an estimate of the variance in their selection criterion (which works  when rank$(\X)<n$), whereas 
Giraud~\cite{Giraud11} introduces a penalized log-likelihood criterion independent of the variance.
Both papers provide oracle risk bounds for the resulting estimators  showing rate-minimax adaptation.

Several recent papers~\cite{Bach08, NeghabanWainwright, RohdeTsybakov, BuneaSheWegkamp11, Kolt11} have investigated another strategy for the low-rank setting. For a positive $\lambda$, the matrix $B_{0}$ is estimated by 
$$\widehat{B}_{\lambda}\in\argmin_{B\in\R^{p\times T}}\Big\{ \|Y-\X B\|_{F}^2+\lambda\sum_{k}s_{k}(B)\Big\}.$$
Translating the work on trace regression of Koltchinskii {\it et al.}~\cite{Kolt11}  into the setting of multivariate regression provides (under some conditions on $\X$)
an oracle bound on the risk of  $\widehat{B}_{\lambda^*}$ with $\lambda^{*}=3s_{1}(X)\big(\sqrt{T}+\sqrt{\mathrm{rank}(X)}\,\big)\sigma$. We also refer to Giraud~\cite{Giraud11b} for a slight variation of this result requiring no condition on the design $\X$. Again, the value of $\lambda^*$ is proportional to $\sigma$. To handle the case of unknown variance, Klopp~\cite{Klopp11} adapts the concept of the square-root Lasso~\cite{squarerootlasso} to this setting and provides an oracle risk bound for the resulting procedure.

\subsection{Nonparametric regression}
In the nonparametric regression model~\eref{modele_generale}, classical estimation procedures include local-polynomial estimators, kernel estimators, basis-projection estimators, $k$-nearest neighbors etc. All these procedures depend on one (or several) tuning parameter(s), whose optimal value(s) scales with the variance $\sigma^2$. $V$-fold \CV\ is widely used in practice for choosing these parameters, but little is known on its theoretical performance. 

The class of linear estimators (including spline smoothing,  Nadaraya estimators, $k$-nearest neighbors, low-pass filters, kernel ridge regression, etc) has attracted some attention in the last years. Some papers have investigated the tuning of some specific family of estimators. For example, Cao and Golubev~\cite{CaoGolubev06} provides a tuning procedure for spline smoothing and Zhang~\cite{Zhang05} analyses in depth kernel ridge regression. Recently, two papers have focused on the tuning of arbitrary linear estimators when the variance $\sigma^2$ is unknown. Arlot and Bach~\cite{ArlotBach09} generalize the slope heuristic to symmetric linear estimators with spectrum in $[0,1]$ and prove an oracle bound for the resulting estimator. Baraud {\it et al.}~\cite{linselect} Section~4 shows that \LinSelect\ can be used for selecting among a (almost) completely arbitrary collection of linear estimators (possibly non-symmetric and/or singular).
Corollary~2 in~\cite{linselect}  provides an oracle bound for the selected estimator under the mild assumption that some effective dimension of the linear estimators is not larger than a fraction of $n$.

%{\color{red} We finally mention the work of Comte and Rozenholc~\cite{CR04} for selecting among piecewise polynomials with variable degrees. 
%(je viens de le bouger l\`a car sa place est en non-param\'etrique. Cependant je viens de jeter un oeil rapide au papier : ils ne montrent rien pour la variance inconnue. Est-ce bien sa place dans cette review sur la variance inconnue?)}

\bibliographystyle{acmtrans-ims}

\bibliography{estimation}

\begin{thebibliography}{}
\ifx \url   \undefined \def \url#1{#1}   \fi

\bibitem{Akaike73}
\textsc{Akaike, H.} (1973).
\newblock Information theory and an extension of the maximum likelihood
  principle.
\newblock In \emph{Second {I}nternational {S}ymposium on {I}nformation {T}heory
  ({T}sahkadsor, 1971)}. Akad\'emiai Kiad\'o, Budapest, 267--281.
\MR{0483125 (58 \#3144)}

\bibitem{Anderson51}
\textsc{Anderson, T.~W.} (1951).
\newblock Estimating linear restrictions on regression coefficients for
  multivariate normal distributions.
\newblock \emph{Ann. Math. Statistics\/}~\emph{22}, 327--351.
\MR{0042664 (13,144f)}

\bibitem{anto2011}
\textsc{Antoniadis, A.} (2010).
\newblock Comments on: {$\ell\sb 1$}-penalization for mixture regression models
  [mr2677722].
\newblock \emph{TEST\/}~\textbf{19},~2, 257--258.
\newblock http://dx.doi.org/10.1007/s11749-010-0198-y.
\MR{2677723}

\bibitem{ArlotBach09}
\textsc{Arlot, S.} \textsc{and} \textsc{Bach, F.} (2009).
\newblock Data-driven calibration of linear estimators with minimal penalties.
\newblock In \emph{Advances in Neural Information Processing Systems 22},
  {Y.~Bengio}, {D.~Schuurmans}, {J.~Lafferty}, {C.~K.~I. Williams}, {and}
  {A.~Culotta}, Eds. 46--54.

\bibitem{2010_ArlotCelisse_StaComp}
\textsc{Arlot, S.} \textsc{and} \textsc{C\'elisse, A.} (2010a).
\newblock Segmentation of the mean of heteroscedastic data via
  cross-validation.
\newblock \emph{Stat. Comput.\/}~\textbf{21},~4, 1--20.
\newblock 10.1007/s11222-010-9196-x,
  http://dx.doi.org/10.1007/s11222-010-9196-x.

\bibitem{ArlotCelisse}
\textsc{Arlot, S.} \textsc{and} \textsc{C\'elisse, A.} (2010b).
\newblock A survey of cross-validation procedures for model selection.
\newblock \emph{Stat. Surv.\/}~\emph{4}, 40--79.
\newblock http://dx.doi.org/10.1214/09-SS054.
\MR{2602303 (2011g:62111)}

\bibitem{2010_ArlotMassart_JMLR}
\textsc{Arlot, S.} \textsc{and} \textsc{Massart, P.} (2010).
\newblock Data-driven calibration of penalties for least-squares regression.
\newblock \emph{J. Mach. Learn. Res.\/}~\emph{10}, 245--279.

\bibitem{groupbach}
\textsc{Bach, F.} (2008a).
\newblock Consistency of the group lasso and multiple kernel learning.
\newblock \emph{J. Mach. Learn. Res.\/}~\emph{9}, 1179--1225.
\MR{2417268 (2010a:68132)}

\bibitem{Bach08}
\textsc{Bach, F.} (2008b).
\newblock Consistency of trace norm minimization.
\newblock \emph{J. Mach. Learn. Res.\/}~\emph{9}, 1019--1048.
\MR{2417263}

\bibitem{baraniuk08}
\textsc{Baraniuk, R.}, \textsc{Davenport, M.}, \textsc{DeVore, R.},
  \textsc{and} \textsc{Wakin, M.} (2008).
\newblock A simple proof of the restricted isometry property for random
  matrices.
\newblock \emph{Constr. Approx.\/}~\textbf{28},~3, 253--263.
\newblock http://dx.doi.org/10.1007/s00365-007-9003-x.
\MR{MR2453366}

\bibitem{Baraud10}
\textsc{Baraud, Y.} (2011).
\newblock Estimator selection with respect to hellinger-type risks.
\newblock \emph{Probab. Theory Related Fields\/}~\textbf{151},~1--2, 353--401.
\newblock 10.1007/s00440-010-0302-y,
  http://dx.doi.org/10.1007/s00440-010-0302-y.

\bibitem{BGH09}
\textsc{Baraud, Y.}, \textsc{Giraud, C.}, \textsc{and} \textsc{Huet, S.}
  (2009).
\newblock Gaussian model selection with an unknown variance.
\newblock \emph{Ann. Statist.\/}~\textbf{37},~2, 630--672.

\bibitem{linselect}
\textsc{Baraud, Y.}, \textsc{Giraud, C.}, \textsc{and} \textsc{Huet, S.}
  (2010).
\newblock Estimator selection in the gaussian setting.
\newblock arXiv:1007.2096v2.

\bibitem{BaronBirgeMassart99}
\textsc{Barron, A.}, \textsc{Birg{\'e}, L.}, \textsc{and} \textsc{Massart, P.}
  (1999).
\newblock Risk bounds for model selection via penalization.
\newblock \emph{Probab. Theory Related Fields\/}~\textbf{113},~3, 301--413.
\newblock http://dx.doi.org/10.1007/s004400050210.
\MR{1679028 (2000k:62049)}

\bibitem{BMM10}
\textsc{Baudry, J.-P.}, \textsc{Maugis, C.}, \textsc{and} \textsc{Michel, B.}
  (2012).
\newblock Slope heuristics: Overview and implementation.
\newblock \emph{Statist. {C}omput.\/}~\textbf{22},~2, 455--470.

\bibitem{squarerootlasso}
\textsc{Belloni, A.}, \textsc{Chernozhukov, V.}, \textsc{and} \textsc{Wang, L.}
  (2011).
\newblock Square-root lasso: Pivotal recovery of sparse signals via conic
  programming.
\newblock \emph{Biometrika\/}~\textbf{98},~4, 791--806.

\bibitem{bickeltsy08}
\textsc{Bickel, P.}, \textsc{Ritov, Y.}, \textsc{and} \textsc{Tsybakov, A.}
  (2009).
\newblock Simultaneous analysis of lasso and {D}antzig selector.
\newblock \emph{Ann. Statist.\/}~\textbf{37},~4, 1705--1732.
\newblock http://dx.doi.org/10.1214/08-AOS620.
\MR{MR2533469}

\bibitem{BM01}
\textsc{Birg{\'e}, L.} \textsc{and} \textsc{Massart, P.} (2001).
\newblock Gaussian model selection.
\newblock \emph{J. Eur. Math. Soc. (JEMS)\/}~\textbf{3},~3, 203--268.
\MR{MR1848946 (2002i:62072)}

\bibitem{massart_pente}
\textsc{Birg{\'e}, L.} \textsc{and} \textsc{Massart, P.} (2007).
\newblock Minimal penalties for {G}aussian model selection.
\newblock \emph{Probab. Theory Related Fields\/}~\textbf{138},~1-2, 33--73.
\newblock http://dx.doi.org/10.1007/s00440-006-0011-8.
\MR{MR2288064 (2008g:62070)}

\bibitem{BuneaSheWegkamp11}
\textsc{Bunea, F.}, \textsc{She, Y.}, \textsc{and} \textsc{Wegkamp, M.~H.}
  (2011).
\newblock {Optimal selection of reduced rank estimators of high-dimensional
  matrices.}
\newblock \emph{Ann. Stat.\/}~\textbf{39},~2, 1282--1309.

\bibitem{tsybakov_agregation_07}
\textsc{Bunea, F.}, \textsc{Tsybakov, A.}, \textsc{and} \textsc{Wegkamp, M.}
  (2007).
\newblock Aggregation for {G}aussian regression.
\newblock \emph{Ann. Statist.\/}~\textbf{35},~4, 1674--1697.
\MR{MR2351101}

\bibitem{candes07}
\textsc{Candes, E.~J.} \textsc{and} \textsc{Tao, T.} (2007).
\newblock The {D}antzig selector: statistical estimation when {$p$} is much
  larger than {$n$}.
\newblock \emph{Ann. Statist.\/}~\textbf{35},~6, 2313--2351.
\MR{MR2382644}

\bibitem{CaoGolubev06}
\textsc{Cao, Y.} \textsc{and} \textsc{Golubev, Y.} (2006).
\newblock On oracle inequalities related to smoothing splines.
\newblock \emph{Math. Methods Statist.\/}~\textbf{15},~4, 398--414 (2007).
\MR{2301659 (2008i:62039)}

\bibitem{chen98}
\textsc{Chen, S.}, \textsc{Donoho, D.}, \textsc{and} \textsc{Saunders, M.}
  (1998).
\newblock Atomic decomposition by basis pursuit.
\newblock \emph{SIAM J. Sci. Comput.\/}~\textbf{20},~1, 33--61.
\newblock http://dx.doi.org/10.1137/S1064827596304010.
\MR{1639094 (99h:94013)}

\bibitem{CR04}
\textsc{Comte, F.} \textsc{and} \textsc{Rozenholc, Y.} (2004).
\newblock A new algorithm for fixed design regression and denoising.
\newblock \emph{Ann. Inst. Statist. Math.\/}~\textbf{56},~3, 449--473.
\newblock http://dx.doi.org/10.1007/BF02530536.
\MR{2095013 (2005e:62081)}

\bibitem{DT08}
\textsc{Dalalyan, A.} \textsc{and} \textsc{Tsybakov, A.} (2008).
\newblock Aggregation by exponential weighting, sharp oracle inequalities and
  sparsity.
\newblock \emph{Machine Learning\/}~\textbf{72},~1-2, 39-- 61.

\bibitem{DevroyeWagner79}
\textsc{Devroye, L.~P.} \textsc{and} \textsc{Wagner, T.~J.} (1979).
\newblock The {$L\sb{1}$} convergence of kernel density estimates.
\newblock \emph{Ann. Statist.\/}~\textbf{7},~5, 1136--1139.
\MR{536515 (80k:62054)}

\bibitem{donoho_compressed}
\textsc{Donoho, D.} (2006).
\newblock Compressed sensing.
\newblock \emph{IEEE Trans. Inform. Theory\/}~\textbf{52},~4, 1289--1306.
\newblock http://dx.doi.org/10.1109/TIT.2006.871582.
\MR{2241189 (2007e:94013)}

\bibitem{donoho_transition}
\textsc{Donoho, D.} \textsc{and} \textsc{Tanner, J.} (2009).
\newblock Observed universality of phase transitions in high-dimensional
  geometry, with implications for modern data analysis and signal processing.
\newblock \emph{Philos. Trans. R. Soc. Lond. Ser. A Math. Phys. Eng.
  Sci.\/}~\textbf{367},~1906, 4273--4293.
\newblock With electronic supplementary materials available online,
  http://dx.doi.org/10.1098/rsta.2009.0152.
\MR{2546388 (2010k:62407)}

\bibitem{lars}
\textsc{Efron, B.}, \textsc{Hastie, T.}, \textsc{Johnstone, I.}, \textsc{and}
  \textsc{Tibshirani, R.} (2004).
\newblock Least angle regression.
\newblock \emph{Ann. Statist.\/}~\textbf{32},~2, 407--499.
\newblock With discussion, and a rejoinder by the authors.
\MR{MR2060166 (2005d:62116)}

\bibitem{fan_scad}
\textsc{Fan, J.} \textsc{and} \textsc{Li, R.} (2001).
\newblock Variable selection via nonconcave penalized likelihood and its oracle
  properties.
\newblock \emph{J. Amer. Statist. Assoc.\/}~\textbf{96},~456, 1348--1360.
\newblock http://dx.doi.org/10.1198/016214501753382273.
\MR{1946581 (2003k:62160)}

\bibitem{Geisser75}
\textsc{Geisser, S.} (1975).
\newblock The predictive sample reuse method with applications.
\newblock \emph{J. Amer. Statist. Assoc.\/}~\emph{70}, 320--328.

\bibitem{sebastien_agregation}
\textsc{Gerchinovitz, S.} (2011).
\newblock Sparsity regret bounds for individual sequences in online linear
  regression.
\newblock \emph{Proceedings of COLT 2011\/}.

\bibitem{Giraud08}
\textsc{Giraud, C.} (2008a).
\newblock Estimation of {G}aussian graphs by model selection.
\newblock \emph{Electron. {J}. {S}tat.\/}~\emph{2}, 542--563.

\bibitem{giraudagregation}
\textsc{Giraud, C.} (2008b).
\newblock Mixing least-squares estimators when the variance is unknown.
\newblock \emph{Bernoulli\/}~\textbf{14},~4, 1089--1107.
\newblock http://dx.doi.org/10.3150/08-BEJ135.
\MR{2543587 (2010k:62274)}

\bibitem{Giraud11}
\textsc{Giraud, C.} (2011a).
\newblock Low rank multivariate regression.
\newblock \emph{Electron. J. Stat.\/}~\emph{5}, 775--799.

\bibitem{Giraud11b}
\textsc{Giraud, C.} (2011b).
\newblock A pseudo-rip for multivariate regression.
\newblock Arxiv:1106.5599v1.

\bibitem{ggmselect}
\textsc{Giraud, C.}, \textsc{Huet, S.}, \textsc{and} \textsc{Verzelen, N.}
  (2012).
\newblock Graph selection with {GGM}select.
\newblock \emph{Stat. Appl. Genet. Mol. Biol.\/}~\textbf{11},~3, 1--50.

\bibitem{Guthery}
\textsc{Guthery, S.~B.} (1974).
\newblock A transformation theorem for one-dimensional {$F$}-expansions.
\newblock \emph{J. Number Theory\/}~\emph{6}, 201--210.
\MR{0342484 (49 \#7230)}

\bibitem{HKT90}
\textsc{Hall, P.}, \textsc{Kay, J.~W.}, \textsc{and} \textsc{Titterington,
  D.~M.} (1990).
\newblock Asymptotically optimal difference-based estimation of variance in
  nonparametric regression.
\newblock \emph{Biometrika\/}~\textbf{77},~3, 521--528.
\newblock http://dx.doi.org/10.1093/biomet/77.3.521.
\MR{1087842 (92d:62042)}

\bibitem{hastie2009}
\textsc{Hastie, T.}, \textsc{Tibshirani, R.}, \textsc{and} \textsc{Friedman,
  J.} (2009).
\newblock \emph{The elements of statistical learning}, Second ed.
\newblock Springer Series in Statistics. Springer, New York.
\newblock Data mining, inference, and prediction,
  http://dx.doi.org/10.1007/978-0-387-84858-7.
\MR{2722294}

\bibitem{2008_Huang}
\textsc{Huang, J.}, \textsc{Ma, S.}, \textsc{and} \textsc{Zhang, C.-H.} (2008).
\newblock Adaptive {L}asso for sparse high-dimensional regression models.
\newblock \emph{Statist. Sinica\/}~\textbf{18},~4, 1603--1618.
\MR{2469326 (2010a:62214)}

\bibitem{HuangZhang2010}
\textsc{Huang, J.} \textsc{and} \textsc{Zhang, T.} (2010).
\newblock The benefit of group sparsity.
\newblock \emph{Ann. Statist.\/}~\textbf{38},~4, 1978--2004.
\newblock http://dx.doi.org/10.1214/09-AOS778.
\MR{2676881 (2011f:62029)}

\bibitem{Huber}
\textsc{Huber, P.} (1981).
\newblock \emph{Robust statistics}.
\newblock John Wiley \& Sons Inc., New York.
\newblock Wiley Series in Probability and Mathematical Statistics.
\MR{606374 (82i:62057)}

\bibitem{Izenman75}
\textsc{Izenman, A.} (1975).
\newblock Reduced-rank regression for the multivariate linear model.
\newblock \emph{J. Multivariate Anal.\/}~\emph{5}, 248--264.
\MR{0373179 (51 \#9381)}

\bibitem{jinups}
\textsc{Ji, P.} \textsc{and} \textsc{Jin, J.} (2010).
\newblock Ups delivers optimal phase diagram in high dimensional variable
  selection.
\newblock http://arxiv.org/abs/1010.5028.

\bibitem{Klopp11}
\textsc{Klopp, O.} (2011).
\newblock High dimensional matrix estimation with unknown variance of the
  noise.
\newblock Arxiv:1112.3055v1.

\bibitem{Kolt11}
\textsc{Koltchinski, V.}, \textsc{Lounici, K.}, \textsc{and} \textsc{Tsybakov,
  A.} (2011).
\newblock Nuclear norm penalization and optimal rates for noisy low rank matrix
  completion.
\newblock \emph{Annals of Statistics\/}~\textbf{39},~5, 2302--2329.

\bibitem{Laurent00}
\textsc{Laurent, B.} \textsc{and} \textsc{Massart, P.} (2000).
\newblock Adaptive estimation of a quadratic functional by model selection.
\newblock \emph{Ann. Statist.\/}~\textbf{28},~5, 1302--1338.

\bibitem{2005_Lebarbier_SignProc}
\textsc{Lebarbier, E.} (2005).
\newblock Detecting multiple change-points in the mean of gaussian process by
  model selection.
\newblock \emph{Signal Processing\/}~\emph{85}, 717--736.

\bibitem{LLW06}
\textsc{Leng, C.}, \textsc{Lin, Y.}, \textsc{and} \textsc{Wahba, G.} (2006).
\newblock A note on the lasso and related procedures in model selection.
\newblock \emph{Statist. Sinica\/}~\textbf{16},~4, 1273--1284.
\MR{2327490}

\bibitem{LB06}
\textsc{Leung, G.} \textsc{and} \textsc{Barron, A.} (2006).
\newblock Information theory and mixing least-squares regressions.
\newblock \emph{IEEE Trans. Inform. Theory\/}~\textbf{52},~8, 3396--3410.
\newblock http://dx.doi.org/10.1109/TIT.2006.878172.
\MR{2242356}

\bibitem{Li87}
\textsc{Li, K.-C.} (1987).
\newblock Asymptotic optimality for {$C\sb p$}, {$C\sb L$}, cross-validation
  and generalized cross-validation: discrete index set.
\newblock \emph{Ann. Statist.\/}~\textbf{15},~3, 958--975.
\newblock http://dx.doi.org/10.1214/aos/1176350486.
\MR{902239 (89c:62112)}

\bibitem{tsybakovgroup}
\textsc{Lounici, K.}, \textsc{Pontil, M.}, \textsc{Tsybakov, A.}, \textsc{and}
  \textsc{van~de Geer, S.} (2011).
\newblock Oracle inequalities and optimal inference under group sparsity.
\newblock \emph{Annals of Statistics\/}~\textbf{39},~4, 2164--2204.

\bibitem{mallows73}
\textsc{Mallows, C.~L.} (1973).
\newblock Some comments on $c_p$.
\newblock \emph{\textit{Technometrics}\/}~\emph{15}, 661--675.

\bibitem{MB06}
\textsc{Meinshausen, N.} \textsc{and} \textsc{B{\"u}hlmann, P.} (2006).
\newblock High-dimensional graphs and variable selection with the lasso.
\newblock \emph{Ann. Statist.\/}~\textbf{34},~3, 1436--1462.
\MR{MR2278363 (2008b:62044)}

\bibitem{MostellerTukey68}
\textsc{Mosteller, F.} \textsc{and} \textsc{Tukey, J.} (1968).
\newblock Data analysis, including statistics.
\newblock In \emph{Handbook of {S}ocial {P}sychology, Vol. 2}, {G.~Lindsey}
  {and} {E.~Aronson}, Eds. Addison-wesley.

\bibitem{NeghabanWainwright}
\textsc{Negahban, S.} \textsc{and} \textsc{Wainwright, M.~J.} (2011).
\newblock Estimation of (near) low-rank matrices with noise and
  high-dimensional scaling.
\newblock \emph{Ann. Statist.\/}~\textbf{39},~2, 1069--1097.
\newblock http://dx.doi.org/10.1214/10-AOS850.
\MR{2816348}

\bibitem{Nishii84}
\textsc{Nishii, R.} (1984).
\newblock Asymptotic properties of criteria for selection of variables in
  multiple regression.
\newblock \emph{Ann. Statist.\/}~\textbf{12},~2, 758--765.
\newblock http://dx.doi.org/10.1214/aos/1176346522.
\MR{740928 (86f:62109)}

\bibitem{bayesianlasso}
\textsc{Park, T.} \textsc{and} \textsc{Casella, G.} (2008).
\newblock The {B}ayesian lasso.
\newblock \emph{J. Amer. Statist. Assoc.\/}~\textbf{103},~482, 681--686.
\newblock http://dx.doi.org/10.1198/016214508000000337.
\MR{2524001}

\bibitem{raskwain09}
\textsc{Raskutti, G.}, \textsc{Wainwright, M.}, \textsc{and} \textsc{Yu, B.}
  (2011).
\newblock Minimax rates of estimations for high-dimensional regression over
  $l^q$ balls.
\newblock \emph{IEEE Trans. Inf. Theory\/}~\textbf{57},~10, 6976--6994.

\bibitem{tsyrig10}
\textsc{Rigollet, P.} \textsc{and} \textsc{Tsybakov, A.} (2011).
\newblock Exponential screening and optimal rates of sparse estimation.
\newblock \emph{Annals of Statistics\/}~\textbf{39},~2, 731--771.

\bibitem{RohdeTsybakov}
\textsc{Rohde, A.} \textsc{and} \textsc{Tsybakov, A.} (2011).
\newblock Estimation of high-dimensional low-rank matrices.
\newblock \emph{Ann. Statist.\/}~\textbf{39},~2, 887--930.
\newblock http://dx.doi.org/10.1214/10-AOS860.
\MR{2816342}

\bibitem{Schwartz78}
\textsc{Schwarz, G.} (1978).
\newblock Estimating the dimension of a model.
\newblock \emph{Ann. Statist.\/}~\textbf{6},~2, 461--464.
\MR{0468014 (57 \#7855)}

\bibitem{Shao93}
\textsc{Shao, J.} (1993).
\newblock Linear model selection by cross-validation.
\newblock \emph{J. Amer. Statist. Assoc.\/}~\textbf{88},~422, 486--494.
\MR{1224373 (94k:62107)}

\bibitem{Shao97}
\textsc{Shao, J.} (1997).
\newblock An asymptotic theory for linear model selection.
\newblock \emph{Statist. Sinica\/}~\textbf{7},~2, 221--264.
\newblock With comments and a rejoinder by the author.
\MR{1466682 (99m:62104)}

\bibitem{Shibata81}
\textsc{Shibata, R.} (1981).
\newblock An optimal selection of regression variables.
\newblock \emph{Biometrika\/}~\textbf{68},~1, 45--54.
\newblock http://dx.doi.org/10.1093/biomet/68.1.45.
\MR{614940 (84a:62103a)}

\bibitem{stadler}
\textsc{St{\"a}dler, N.}, \textsc{B{\"u}hlmann, P.}, \textsc{and}
  \textsc{van~de Geer, S.} (2010).
\newblock {$\ell\sb 1$}-penalization for mixture regression models.
\newblock \emph{TEST\/}~\textbf{19},~2, 209--256.
\newblock http://dx.doi.org/10.1007/s11749-010-0197-z.
\MR{2677722}

\bibitem{Stone74}
\textsc{Stone, M.} (1974).
\newblock Cross-validatory choice and assessment of statistical predictions.
\newblock \emph{J. Roy. Statist. Soc. Ser. B\/}~\emph{36}, 111--147.
\newblock With discussion by G. A. Barnard, A. C. Atkinson, L. K. Chan, A. P.
  Dawid, F. Downton, J. Dickey, A. G. Baker, O. Barndorff-Nielsen, D. R. Cox,
  S. Giesser, D. Hinkley, R. R. Hocking, and A. S. Young, and with a reply by
  the authors.
\MR{0356377 (50 \#8847)}

\bibitem{sundiscussion}
\textsc{Sun, T.} \textsc{and} \textsc{Zhang, C.-H.} (2010).
\newblock Comments on: {$\ell\sb 1$}-penalization for mixture regression models
  [mr2677722].
\newblock \emph{TEST\/}~\textbf{19},~2, 270--275.
\newblock http://dx.doi.org/10.1007/s11749-010-0201-7.
\MR{2677726}

\bibitem{sun}
\textsc{Sun, T.} \textsc{and} \textsc{Zhang, C.-H.} (2011).
\newblock Scaled sparse linear regression.
\newblock {\texttt arXiv:1104.4595}.

\bibitem{tiblasso}
\textsc{Tibshirani, R.} (1996).
\newblock Regression shrinkage and selection via the lasso.
\newblock \emph{J. Roy. Statist. Soc. Ser. B\/}~\textbf{58},~1, 267--288.
\MR{1379242 (96j:62134)}

\bibitem{2005_TSR_JRSSB}
\textsc{Tibshirani, R.}, \textsc{Saunders, M.}, \textsc{Rosset, S.},
  \textsc{Zhu, J.}, \textsc{and} \textsc{Knight, K.} (2005).
\newblock Sparsity and smoothness via the fused lasso.
\newblock \emph{J. R. Stat. Soc. Ser. B Stat. Methodol.\/}~\textbf{67},~1,
  91--108.
\newblock http://dx.doi.org/10.1111/j.1467-9868.2005.00490.x.
\MR{2136641}

\bibitem{geer_condition}
\textsc{van~de Geer, S.} \textsc{and} \textsc{B{\"u}hlmann, P.} (2009).
\newblock On the conditions used to prove oracle results for the {L}asso.
\newblock \emph{Electron. J. Stat.\/}~\emph{3}, 1360--1392.
\newblock http://dx.doi.org/10.1214/09-EJS506.
\MR{2576316 (2011c:62231)}

\bibitem{VertNIPS2010}
\textsc{Vert, J.-P.} \textsc{and} \textsc{Bleakley, K.} (2010).
\newblock Fast detection of multiple change-points shared by many signals using
  group lars.
\newblock In \emph{Advances in Neural Information Processing Systems 23},
  {J.~Lafferty}, {C.~K.~I. Williams}, {J.~Shawe-Taylor}, {R.~Zemel}, {and}
  {A.~Culotta}, Eds. 2343--2351.

\bibitem{verzelen_regression}
\textsc{Verzelen, N.} (2010).
\newblock High-dimensional gaussian model selection on a gaussian design.
\newblock \emph{Ann. Inst. H. Poincar\'e Probab. Statist.\/}~\textbf{46},~2,
  480--524.

\bibitem{Vminimax}
\textsc{Verzelen, N.} (2012).
\newblock Minimax risks for sparse regressions: Ultra-high-dimensional
  phenomenons.
\newblock \emph{Electron. J. Stat.\/}~\emph{6}, 38--90.

\bibitem{wain_minimax2}
\textsc{Wainwright, M.} (2009).
\newblock Information-theoretic limits on sparsity recovery in the
  high-dimensional and noisy setting.
\newblock \emph{IEEE Trans. Inform. Theory\/}~\textbf{55},~12, 5728--5741.
\newblock http://dx.doi.org/10.1109/TIT.2009.2032816.
\MR{MR2597190}

\bibitem{yeminimax}
\textsc{Ye, F.} \textsc{and} \textsc{Zhang, C.-H.} (2010).
\newblock Rate minimaxity of the {L}asso and {D}antzig selector for the
  {$\ell\sb q$} loss in {$\ell\sb r$} balls.
\newblock \emph{J. Mach. Learn. Res.\/}~\emph{11}, 3519--3540.
\MR{2756192}

\bibitem{groupyuan}
\textsc{Yuan, M.} \textsc{and} \textsc{Lin, Y.} (2006).
\newblock Model selection and estimation in regression with grouped variables.
\newblock \emph{J. R. Stat. Soc. Ser. B Stat. Methodol.\/}~\textbf{68},~1,
  49--67.
\newblock http://dx.doi.org/10.1111/j.1467-9868.2005.00532.x.
\MR{2212574}

\bibitem{zhangMC+}
\textsc{Zhang, C.-H.} (2010).
\newblock Nearly unbiased variable selection under minimax concave penalty.
\newblock \emph{Ann. Statist.\/}~\textbf{38},~2, 894--942.
\newblock http://dx.doi.org/10.1214/09-AOS729.
\MR{2604701 (2011d:62211)}

\bibitem{Zhang05}
\textsc{Zhang, T.} (2005).
\newblock Learning bounds for kernel regression using effective data
  dimensionality.
\newblock \emph{Neural Comput.\/}~\textbf{17},~9, 2077--2098.
\newblock http://dx.doi.org/10.1162/0899766054323008.
\MR{2175849 (2006d:62062)}

\bibitem{zhangFwdBwd}
\textsc{Zhang, T.} ({2011}).
\newblock {Adaptive Forward-Backward Greedy Algorithm for Learning Sparse
  Representations}.
\newblock \emph{IEEE Trans. Inform. Theory\/}~\textbf{{57}},~{7}, {4689--4708}.

\bibitem{zou_adaptive}
\textsc{Zou, H.} (2006).
\newblock The adaptive lasso and its oracle properties.
\newblock \emph{J. Amer. Statist. Assoc.\/}~\textbf{101},~476, 1418--1429.
\MR{MR2279469 (2008d:62024)}

\bibitem{zou05}
\textsc{Zou, H.} \textsc{and} \textsc{Hastie, T.} (2005).
\newblock Regularization and variable selection via the elastic net.
\newblock \emph{J. R. Stat. Soc. Ser. B Stat. Methodol.\/}~\textbf{67},~2,
  301--320.
\MR{MR2137327}

\end{thebibliography}

\appendix
\numberwithin{equation}{section}

\section{A note on \BIC\ type criteria}\label{section_appendix_BIC}

The \BIC\ criterion has been initially introduced~\cite{Schwartz78} to select an estimator among a collection of constrained maximum likelihood estimators. Nevertheless, modified versions of this criterion are often used for tuning more general estimation procedures. The purpose of this appendix is to illustrate why we  advise against this approach in a high-dimensional setting.

\begin{defi}{\bf A Modified \BIC\ criterion}.
Suppose we are given a collection  $(\widehat{\beta}_{\lambda})_{\lambda\in\Lambda}$ of estimators  depending on a tuning parameter $\lambda\in \Lambda$. For any $\lambda\in\Lambda$, we consider $\widehat{\sigma}_{\lambda}^2=\|Y-{\bf X}\widehat{\beta}_{\lambda}\|_2^2/n$, and define the modified $\BIC$ 
\begin{equation}\label{critere_BIC_generalise}
\widehat{\lambda}\in \argmin_{\lambda\in \widehat{\Lambda}} \ac{-2\mathbf{L}_n(\hbeta_{\lambda},\widehat{\sigma}_\lambda)+ \log(n)\|\widehat{\beta}_{\lambda}\|_0} ,
\end{equation}
where $\mathbf{L}_n$ is the log-likelihood and $\widehat{\Lambda}=\ac{\lambda\in\Lambda\,:\ \|\widehat{\beta}_{\lambda}\|_0\leq n/2}$.
\end{defi}
Sometimes, the $\log(n)$ term is replaced by $\log(p)$. Replacing $\Lambda$ by $\widehat{\Lambda}$ allows to avoid trivial estimators.
First, we would like to emphasize that there is {\it no} theoretical warranty that the selected estimator does not overfit in a high-dimensional setting. In practice, using this criterion often leads to overfitting. Let us illustrate this with a simple experiment.\smallskip

\noindent
\textbf{\textsf{Setting.}} We consider the model
\begin{equation}\label{modele_sequence_gaussien}
Y_{i}=\beta_{0,i}+\eps_{i},\ i=1,\ldots,n\ ,
\end{equation}
with $\eps\sim\mathcal{N}(0,\sigma^2 I_n)$ so that $p=n$ and ${\bf X}=I_n$. Here, we fix $n=10000$, $\sigma=1$ and $\beta_0=0_n$. \smallskip

\noindent
\textbf{\textsf{Methods.}} We apply the modified \BIC\ criterion to tune the Lasso~\cite{tiblasso}, SCAD~\cite{fan_scad} and  the hard thresholding estimator.
The hard thresholding estimator $\widehat{\beta}^{HT}_{\lambda}$ is defined for any $\lambda>0$ by $[\widehat{\beta}^{HT}_{\lambda}]_i=Y_i\mathbf
{1}_{|Y_i|\geq \lambda}$. Given $\lambda>0$ and $a>2$, the SCAD estimator $\widehat{\beta}^{SCAD}_{\lambda,a}$ is defined as the minimizer of the penalized criterion
$
 \|Y-{\bf X}\beta\|_2^2+ \sum_{i=1}^n p_{\lambda}(|\beta_i|)\ ,$
where for $x>0$, $$p'_{\lambda}(x)= \lambda\mathbf{1}_{x\leq \lambda}+ (a\lambda-x)_{+}\mathbf{1}_{x>\lambda}/(a-1)\ .$$
For the sake of simplicity we fix $a=3$. We note $\widehat{\beta}^{L;\BIC}$, $\widehat{\beta}^{SCAD;\BIC}_a$, and $\widehat{\beta}^{HT;\BIC}$ for  the Lasso, hard thresholding, and SCAD estimators selected by the modified \BIC\ criterion.  
\smallskip

 \noindent 
\textbf{\textsf{Results.}}
 We have realized $N=200$ experiments. For each of these experiments, the estimator $\widehat{\beta}^{L;\BIC}$, $\widehat{\beta}^{SCAD,\BIC}_a$ and $\widehat{\beta}^{HT;\BIC}$ are computed. The mean number of non-zero components and the estimated risk $\mathcal{R}[\widehat{\beta}^{*;\BIC};0_n]$ are reported in Table~\ref{tableau_dimension}.

\begin{Table} [h] 
 
\caption{Estimated risk and Estimated number of non zero components for $\widehat{\beta}^{L;\BIC}$, $\widehat{\beta}^{SCAD;\BIC}$, and $\widehat{\beta}^{HT;\BIC}$. \label{tableau_dimension} }

\begin{center}

\begin{tabular}{c|c|c|c|}
  &{\normalsize LASSO} & {\normalsize SCAD}& {\normalsize Hard Thres.}  \\\hline &&&  \\
{\normalsize$\widehat{\mathcal{R}}[\widehat{\beta}^{*;\BIC};0_p]$} &{\normalsize 4.6$\times 10^{-2}$}&{\normalsize 1.6$\times 10^1$}   & {\normalsize 3.0$\times 10^{2}$} \\\hline &&&\\ 
{\normalsize Mean of $ {\|\widehat{\beta}^{*;\BIC}\|_0}$}& {\normalsize 0.025}& {\normalsize 86.9} & {\normalsize 28.2}\\ \hline
\end{tabular}
 
\end{center}
\end{Table}
Obviously, the SCAD and hard Thresholding methods select too many irrelevant variables when they are tuned with \BIC. Moreover, their risks are quite high. Intuitively, this is due to the fact that the  $\log(n)$ (or $\log(p)$) term in the \BIC\ penalty is too small in this high-dimensional setting ($n=p$). 

For the Lasso estimator, a very specific phenomenon occurs due to the soft thresholding effect. In the discussion of~\cite{lars}, Loubes and Massart advocate that soft thresholding estimators penalized by Mallows' $C_p$~\cite{mallows73} penalties should yield good results, while hard thresholding estimators penalized by Mallows' $C_p$ are known to highly overfit. This strange behavior is due to the bias of the soft thresholding estimator. Nevertheless, Loubes and Massart' arguments have been developed for an orthogonal design. In fact, there is no non-asymptotic justification that  the Lasso tuned by \BIC\ or \AIC\ performs well for general designs ${\bf X}$.\smallskip

\noindent 
\textbf{\textsf{Conclusion.}} The use of the modified \BIC\ criterion to tune estimation procedures in a high-dimensional setting is not supported by theoretical results. It is proved to overfit in the case of thresholding estimators~\cite[Sect. 3.2.2]{BGH09}. 
Empirically, \BIC\ seems to  overfit except for the  Lasso. We advise the practitioner to avoid \BIC\ (and \AIC) when $p$ is at least of the same order as $n$. For instance, \LinSelect\ is supported by non-asymptotic arguments and by empirical results~\cite{linselect} in contrast to \BIC.

\section{Minimax Adaptive procedures}\label{section_appendix_minimax}

In this section, we detail procedures that are minimax adaptive to the sparsity $k$ simultaneously for all designs ${\bf X}$ in the sense of (\ref{minoration_minimax_prediction}). In most settings, these procedures are not of practical interest as they are intractable for large $p$. We present them as theoretical benchmarks to assess the quality of fast procedures.

Given a subspace $S$ of $\mathbb{R}^n$, we define $\widehat{\beta}^{\perp}_S$ as a least-squares estimator of $\beta_0$ such that ${\bf X}\beta$ is included in $S$:
\[\widehat{\beta}^{\perp}_S\in\argmin_{\beta\in\R^p,\ {\bf X}\beta\in S}\|Y-{\bf X}\beta\|_2^2\ .\]
We consider the collections of subspaces:
\begin{eqnarray*}
\S_1&=& \Big\{S=\range(\X_{\J}),\ \J\subset \ac{1,\ldots,p}\setminus\{\emptyset\},\quad 2|\J|[1+\log(p/|\J|)]\leq n\Big\}\\
&&\ \bigcup \ \range(\X_{\{1,\ldots p\}})\,, \\
\S_2&=&\Big\{S=\range(\X_{\J}),\ \J\subset \ac{1,\ldots,p}\setminus\{\emptyset\},\quad |\J|\leq (n-1)/4\Big\}\ .
\end{eqnarray*}
 Finally, we note $k^*:= \max\{k:2k[1+\log(p/k)]\leq n\}$. To simplify the presentation, we assume throughout this section that $n\leq p$ and that $\mathrm{Rank}({\bf X})> k^*$.

\subsection{Known variance}\label{section_appendix_minimax_1}

\noindent{\bf A penalization strategy}.
The model selection paradigm aims at selecting an estimator
$\widehat{\beta}_{\widehat{S}}$ with the smallest possible risk. One
strategy to tackle the selection problem amounts to minimizing a
least-squares criterion penalized by the "complexity" of the
collection of models under consideration. We select $\widehat{S}^{BM}$ as one minimizer over $S\in\S_1$ of the following criterion
\begin{eqnarray*}
\| Y-\Pi_S Y\|_2^2+\left\{ \begin{array}{cl}                  
4\dim(S)\left[4 +\log\left(\frac{p}{\dim(S)}\right)\right]\sigma^2 & \text{if } 
\dim(S)\leq k^*\\
2n\sigma^2 &\text{if }  \dim(S)=\mathrm{Rank}({\bf X})\,,
                  \end{array}\right.
\end{eqnarray*}
We write $\widetilde{\beta}^{BM}:= \widehat{\beta}^{\perp}_{\widehat{S}^{BM}}$. More general forms of penalties are discussed in \cite{BM01}.\\

\noindent{\bf An aggregation strategy}. In contrast to model selection, model aggregation aims at mixing a collection of estimators.
Following, Leung and Barron~\cite{LB06}, we mix the least-squares estimators $\widehat{\beta}_S$ in the following way
\[\widetilde{\beta}^{LB}:=\sum_{S\in\S_1}\omega_S \widehat{\beta}^{\perp}_S\ ,\]
where the weights $\omega_S$ sum to one and for any $S\in \S_1$, $\omega_S$ is proportional to 
\begin{eqnarray*}
\exp\left[-\frac{\|Y-\Pi_S Y\|_2^2+2\sigma^2\dim(S)}{4\sigma^2}\right]\times  \left\{ \begin{array}{cl} \left[k^*\binom{\dim(S)}{p}\right]^{-1} & \text{if } 
\dim(S)\leq k^*\\
1 &\text{if }  \dim(S)=\mathrm{Rank}({\bf X}).
                  \end{array}\hspace{-1cm}\right.
\end{eqnarray*}
 We refer to~\cite{LB06} for more general forms of the aggregation procedures.\\

\noindent{\bf Risk bounds}. In the next proposition, we state that $\widetilde{\beta}^{BM}$ and $\widetilde{\beta}^{LB}$ are minimax adaptive to the sparsity for all designs ${\bf X}$ in the sense of (\ref{minoration_minimax_prediction}).

\begin{prte}
There exist numerical constants $C_1$ and $C_2$ such that the following holds.
For any design ${\bf X}$, any $k\in \{1,\ldots, n\}$ and any vector $\beta_0$   such that
  $\|\beta_0\|_0=k$,  we have
 \begin{eqnarray*}
  \mathcal{R}\left[\widetilde{\beta}^{BM};\beta_0\right]&\leq& C_1\left[k\left(1+\log\left(\frac{p}{k}\right)\right)\wedge n\right]\sigma^2\ ,\\
\mathcal{R}\left[\widetilde{\beta}^{LB};\beta_0\right]&\leq& C_2\left[k\left(1+\log\left(\frac{p}{k}\right)\right)\wedge n\right]\sigma^2\ .
 \end{eqnarray*}
\end{prte}
These two risk bounds derive straightforwardly from the aforementioned work~\cite{BM01,LB06}.

\subsection{Unknown variance}\label{section_appendix_minimax_2}

For any set $S\in\S_2$, we set the following  measure of complexity $\Delta(S)$
\begin{eqnarray*}
\Delta(S)&=&\log\binom{p}{\dim(S)}+\log(\dim(S))\ ,
\end{eqnarray*}
and we take the same penalty term $\pen(S)$ as for \LinSelect\ (see Appendix \ref{appendix-linselect_details}).
Baraud et al. \cite{BGH09}  consider the model selection estimators $\widetilde{\beta}^{BGH}:= \widehat{\beta}^{\perp}_{\widehat{S}^{BGH}}$ with
\begin{eqnarray*}
 \widehat{S}^{BGH}:=\argmin_{S\in \S_2}\| Y-\Pi_S Y\|_2^2\left[1+ \frac{\pen(S)}{n-\dim(S)}\right].
\end{eqnarray*}
The first risk bound only covers the (non-ultra) high-dimensional setting.
\begin{prte}\label{prte_risque_bgh}
There exists some numerical constant $C$ such that the following holds.
For any design ${\bf X}$ and any vector $\beta_0$, we have
\begin{eqnarray*}
 \mathcal{R}\left[\widetilde{\beta}^{BGH};\beta_0\right]&\leq& C\hspace{-0.4cm}\inf_{\begin{scriptsize}\begin{array}{c}\beta\in\mathbb{R}^p\\ \ \|\beta\|_0\leq \frac{n}{2\log(p)}\end{array}\end{scriptsize}}
 \left\{ \|{\bf X}\left(\beta-\beta_0\right)\|_2^2  + \|\beta\|_0\left[1+\log\left(\frac{p}{\|\beta\|_0}\right)\right]\sigma^2\right\}.
\end{eqnarray*}
\end{prte}
Proposition \ref{prte_risque_bgh} is a straightforward consequence of Corollary 1 in \cite{BGH09}. It shows that simultaneous adaptation to the variance and the sparsity is possible if we restrict ourselves to a non-ultra high-dimensional setting. The next proposition complements the risk upper bound of Proposition \ref{prte_adaptation_impossible}. Consider $\widetilde{\beta}^{(n)}$ as a least-squares estimator of $\beta_0$ over $\mathbb{R}^n$.

\begin{prte}\label{prte_risque_bgh2}
There exist numerical constants $C$, $C_1$, and $C_2$ such that the following holds.
For any design ${\bf X}$, any $\sigma>0$, and any vector $\beta_0\in\mathbb{R}^p$, we have
\begin{eqnarray*}
 \mathcal{R}\left[\widetilde{\beta}^{(n)};\beta_0\right]&\leq& Cn\sigma^2.
\end{eqnarray*}
For any design ${\bf X}$, any  $\sigma>0$, any $k\in \{1,\ldots, (n-1)/4\}$ and any vector $\beta_0\in\mathbb{R}^p$ such that
  $\|\beta_0\|_0=k$,  we have
\begin{eqnarray*}
 \mathcal{R}\left[\widetilde{\beta}^{BGH};\beta_0\right]&\leq& C_1 k\log\left(\frac{p}{k}\right)\exp\left[C_2\frac{k}{n}\log\left(\frac{p}{k}\right)\right]\sigma^2.
\end{eqnarray*}
\end{prte}
The first bound is straightforward while the second bound derives from \cite{BGH09}.

\section{Complements on LinSelect}\label{appendix-linselect}
\subsection{More details  on the selection procedure}\label{appendix-linselect_details}
The penalty $\pen(S)$  involved in the 
\LinSelect\ criterion~(\ref{eq:linselect}) is defined by
$\pen(S)=1.1\,\pen_{\Delta}(S)$
where $\pen_{\Delta}(S)$
is the unique solution of 
$$ \E\cro{\pa{U-{\pen_{\Delta}(S)\over n-\dim(S)}V}_{+}}=e^{-\Delta(S)} $$
where
$U$ and $V$ are two independent chi-square random variables with $\dim(S)+1$ and $n-\dim(S)-1$ degrees of freedom respectively. It is also the solution in $x$ of 
\begin{eqnarray*}
\lefteqn{e^{-\Delta(S)}=}\\
&&(D+1)\P\pa{F_{D+3,N-1}\geq x{N-1\over N(D+3)}}-x{N-1\over N}\P\pa{F_{D+1,N+1}\geq x{N+1\over N(D+1)}}
\end{eqnarray*}
where  $D=\dim(S)$, $N=n-\dim(S)$ and $F_{d,r}$ is a Fisher random variable with $d$ and $r$ degrees of freedom. 
\smallskip

Proposition 4 in~\cite{BGH09} ensures the following upper-bound on $\pen_{\Delta}(S)$. For any $0<\kappa<1$, there exists a constant $C_{\kappa}>1$ such that for any $S\in\S$ fulfilling $1\leq \dim(S)\vee \Delta(S)\leq \kappa n$ we have
$$\pen_{\Delta}(S)\leq C_{\kappa} \big(\dim(S)\vee\Delta(S)\big).$$
Conversely, Lemma~\ref{lemma_minoration_penalite} in Appendix~\ref{proof:ORACLE-LINSELECT} ensures that $\pen_{\Delta}(S)\geq 2\Delta(S)+\dim(S)-C$ for some constant $C\geq 0$.

\subsection{A general risk bound for LinSelect}
We set 
\begin{equation}\label{eq:sum}
\Sigma=\sigma^2\sum_{S\in\S}e^{-\Delta(S)}.
\end{equation}
The following proposition gives a risk bound when selecting $\widehat\lambda$ by minimizing (\ref{eq:linselect}).
%%%%%
\begin{prte}\label{ORACLE-LINSELECT}
Assume that $1\leq\dim(S)\leq n/2-1$ and $\Delta(S)\leq 2n/3$ for all $S\in\S$. 
Then, there exists a constant $C>1$ such that  for any minimizer $\widehat{\lambda}$ of the Criterion~\eref{eq:linselect}, we have
\begin{eqnarray*}\label{eq:risque_esp}
\lefteqn{C^{-1}\mathcal{R}\left[\widehat{\beta}_{\widehat{\lambda}};\beta_0\right]}\\
& \hspace{-0.4cm}\leq &
\hspace{-0.3cm}\E\cro{\inf_{\lambda\in\Lambda} \ac{\|{\bf X}\hbeta_{\lambda}-{\bf X}\beta_{0}\|_2^2+
 \inf_{S\in\widehat \S}\ac{\|\X\hbeta_{\lambda}-\Pi_{S}\X\hbeta_{\lambda}\|^2_{2}+[\Delta(S)\vee \dim(S)]\sigma^2}}}+\Sigma.
\end{eqnarray*}
Furthermore, with probability larger than $1-e^{-C_0n}-C_1\sum_{S\in\mathbb{S}}e^{-C_2[\Delta(S)\wedge n]}e^{-\Delta(S)}$, we have for some $C>1$
\begin{eqnarray*}\label{eq:risque_proba}
C^{-1}\lefteqn{\norm{{\bf X}\beta_{0}-{\bf X}\widehat \beta_{\widehat \lambda}}^{2}_{2}}\nonumber\\ 
&\le&  \inf_{\lambda\in\Lambda} \ac{\|{\bf X}\hbeta_{\lambda}-{\bf X}\beta_{0}\|_2^2+
 \inf_{S\in\widehat \S}\ac{\|\X\hbeta_{\lambda}-\Pi_{S}\X\hbeta_{\lambda}\|^2_{2}+[\Delta(S)\vee \dim(S)]\sigma^2}}.
\end{eqnarray*}
\end{prte}
%%%%%%
The first part of Proposition~\ref{ORACLE-LINSELECT} is a slight variation of  Theorem 1 in~\cite{linselect}. We refer to  the Appendix~\ref{proof:risque-esp} for a sketch of the proof of this result. The second part is proved in  Appendix~\ref{proof:risque-proba}.

%\newpage

%\centerline{\Large SUPPLEMENTARY MATERIAL}

\section{Proof of Proposition~\ref{ORACLE-LINSELECT}}\label{proof:ORACLE-LINSELECT}
\subsection{Proof of the first part of Proposition \ref{ORACLE-LINSELECT}}\label{proof:risque-esp}
In this section $C$ denotes a constant whose value may vary from line to line.
We also use in this section the notations $\|.\|$ for $\|.\|_{2}$, 
%$\pen(S)=1.1\pen_{\Delta}(S)$, 
$f_{0}=\X\beta_{0}$ and $\hf_{\lambda}=\X\widehat\beta_{\lambda}$. Finally, for any $S\in\S$, we write $\overline S$ for the linear space generated by $S$ and $f_{0}$.
Let $(\widehat \lambda,S_{*})$ be any minimizer over $\Lambda\times\hS$ of
\[
\crit(\lambda,S)=\norm{Y-\Pi_{S}\widehat f_{\lambda}}^{2}+{1\over 2}\norm{\widehat f_{\lambda}-\Pi_{S}\widehat f_{\lambda}}^{2}+\pen(S)\widehat{\sigma}^2_{S}.
\]
 From $\crit(\widehat\lambda,S_{*})\leq\crit(\lambda,S)$ and simple algebra, we get for any $K>1$, $\lambda\in\widehat{\Lambda}$ and $S\in\hS$
\begin{eqnarray*}
\lefteqn{\norm{f_{0}-\Pi_{S_{*}}\widehat f_{\widehat \lambda}}^{2}+{1\over 2}\norm{\widehat f_{\widehat \lambda}-\Pi_{S_{*}}\widehat f_{\widehat \lambda}}^{2}}\\
%&\le& \norm{f-\Pi_{S}\widehat f_{\lambda}}^{2}+2\<\eps,\Pi_{S_{*}}\widehat f_{\widehat \lambda}-\Pi_{S} \widehat f_{\lambda}\>+{1\over 2}\norm{\widehat f_{\lambda}-\Pi_{S}\widehat f_{\lambda}}^{2}\\
%&& +\pen(S)\widehat{\sigma}^2-\widehat{\sigma}^2\pen(S_{*})\\
%&\le& \norm{f-\Pi_{S}\widehat f_{\lambda}}^{2}+2\<\eps,\Pi_{S_{*}}\widehat f_{\widehat \lambda}-f+f-\Pi_{S}\widehat f_{\lambda}\>+{1\over 2}\norm{\widehat f_{\lambda}-\Pi_{S}\widehat f_{\lambda}}^{2}\\
%&& +\pen(S)\widehat{\sigma}^2-\widehat{\sigma}^2\pen(S_{*})\\
&\le& \norm{f_{0}-\Pi_{S}\widehat f_{\lambda}}^{2}+{1\over 2}\norm{\widehat f_{\lambda}-\Pi_{S}\widehat f_{\lambda}}^{2}+ 2\pen(S)\widehat{\sigma}^2_{S} \\
&&\ +\ 2\<\eps,\Pi_{S_{*}}\widehat f_{\widehat \lambda}-f_{0}\>-\pen(S_{*})\widehat{\sigma}^2_{S_{*}}\ +\ 2\<\eps, f_{0}-\Pi_{S}\widehat f_{ \lambda}\>-\pen(S)\widehat{\sigma}^2_{S}.
%\end{eqnarray*}
%For $\lambda\in \Lambda$ and $S\in\S$, let us set
%$u_{\lambda,S}=\pa{\Pi_{S}\widehat
% f_{\lambda}-f}/\norm{\Pi_{S}\widehat f_{\lambda}-f}$ if
%$\Pi_{S}\widehat  f_{\lambda}\ne f$ and $u_{\lambda,S}=0$ otherwise. 
%For all $\lambda$ and $S$, we have $u_{\lambda,S}\in\overline S$ and
%\begin{eqnarray*}
%\lefteqn{\norm{f-\Pi_{S_{*}}\widehat f_{\widehat \lambda}}^{2}+{1\over 2}\norm{\widehat f_{\widehat \lambda}-\Pi_{S_{*}}\widehat f_{\widehat \lambda}}^{2}}\\
%&\le& \norm{f-\Pi_{S}\widehat f_{\lambda}}^{2}+ {1\over 2}\norm{\widehat f_{\lambda}-\Pi_{S}\widehat f_{\lambda}}^{2}+ 2\pen(S)\widehat{\sigma}^2{+2\delta}\\
%&& +\ 2\ab{\<\eps,u_{\widehat\lambda,S_{*}}\>}\norm{\Pi_{S_{*}}\widehat f_{\widehat \lambda}-f}-\widehat{\sigma}^2\pen(S_{*})\\
%&& +\ 2\ab{\<\eps,u_{\lambda,S}\>}\norm{\Pi_{S}\widehat
% f_{\lambda}-f}-\pen(S)\widehat{\sigma}^2.
\\
&\le& \norm{f_{0}-\Pi_{S}\widehat f_{\lambda}}^{2}+{1\over 2}\norm{\widehat f_{\lambda}-\Pi_{S}\widehat f_{\lambda}}^{2}+2\pen(S)\widehat{\sigma}^2_{S}\\
&& +\ K^{-1}\norm{f_{0}-\Pi_{S_{*}}\widehat f_{\widehat \lambda}}^{2}+K\norm{\Pi_{\bar S_{*}}\eps}^{2}-\pen(S_{*})\widehat{\sigma}^2_{S_{*}}\\
&& +\ K^{-1}\norm{f_{0}-\Pi_{S}\widehat f_{\lambda}}^{2}+K\norm{\Pi_{\bar S}\eps}^{2}-\pen(S)\widehat{\sigma}^2_{S},
\end{eqnarray*}
%\begin{color}{cyan}
the second inequality following from $2\<f,g\> \leq K^{-1} \norm{f}^{2} + K \norm{g}^{2}$.
Introducing the notation 
\[
\tilde  \Sigma=2\sum_{S\in\S}\pa{K\norm{\Pi_{\overline S}\eps}^{2}-{\pen(S)\over n-\dim(S)}\norm{Y-\Pi_{\overline S}Y}^{2}}_{+},
\]
we can reformulate the above bound as
\begin{eqnarray}
\lefteqn{\pa{{2}+{1\over 1-K^{-1}}}^{-1}\norm{f_{0}-\widehat f_{\widehat \lambda}}^{2}}\nonumber &&\\ &\le& (1-K^{-1})\norm{f_{0}-\Pi_{S_{*}}\widehat f_{\widehat \lambda}}^{2}+{1\over 2}\norm{\widehat f_{\widehat \lambda}-\Pi_{S_{*}}\widehat f_{\widehat \lambda}}^{2}\nonumber\\
&\le& (1+K^{-1})\norm{f_{0}-\Pi_{S}\widehat f_{\lambda}}^{2}+{1\over 2}\norm{\widehat f_{\lambda}-\Pi_{S}\widehat f_{\lambda}}^{2}+2\pen(S)\widehat{\sigma}^2_{S}+\tilde  \Sigma. \nonumber\\ \label{eq:intermediaire}
\end{eqnarray}
For any $S\in\hS$ we have $\dim(S)\le n/2-1$ and $\Delta(S)\leq 2n/3$. Therefore, according to Proposition~4 in~\cite{BGH09} we have
$\pen(S)\leq C [\dim(S)\vee \Delta(S)]$ and then
\begin{eqnarray*}
\pen(S)\widehat\sigma^2_{S}&=&{\pen(S)\over n-\dim(S)}\|Y-\Pi_{S}Y\|^2
\ \leq\ {\pen(S)\over n-\dim(S)}\|Y-\Pi_{S}\hf_{\lambda}\|^2\\
&\hspace{-2.3cm}\leq&\hspace{-1.3cm} 3 {\pen(S)\over n-\dim(S)}\pa{\|\eps\|^2+\|f_{0}-\hf_{\lambda}\|^2+\|\hf_{\lambda}-\Pi_{S}\hf_{\lambda}\|^2}\\
%&\leq& 3 {\pen_{\Delta}(S)\over n-\dim(S)}\pa{\pa{\|\eps\|^2-2n\sigma^2}_{+}+\|f_{0}-\hf_{\lambda}\|^2+\|\hf_{\lambda}-\Pi_{S}\hf_{\lambda}\|^2}\\
%&&+\ {6n\over n-\dim(S)}\pen_{\Delta}(S)\sigma^2.
&\hspace{-2.3cm}\leq&\hspace{-1.3cm} C \pa{[\dim(S)\vee \Delta(S)]\sigma^2+\pa{\|\eps\|^2-2n\sigma^2}_{+}+\|f_{0}-\hf_{\lambda}\|^2+\|\hf_{\lambda}-\Pi_{S}\hf_{\lambda}\|^2},
\end{eqnarray*}
where $C$ is a positive constant.
Combining this bound with (\ref{eq:intermediaire}) and 
$$(1+K^{-1})\norm{f_{0}-\Pi_{S}\widehat f_{\lambda}}^{2}+{1\over 2}\norm{\widehat f_{\lambda}-\Pi_{S}\widehat f_{\lambda}}^{2} \le 
4\norm{f_{0}-\widehat f_{\lambda}}^{2}+5\norm{\widehat f_{\lambda}-\Pi_{S}\widehat f_{\lambda}}^{2}$$
we finally obtain that for any $\lambda\in\Lambda$ and $S\in\hS$
\begin{equation}\label{eq:corollaire_generale}
C^{-1}\norm{f_{0}-\widehat f_{\widehat \lambda}}^{2} \le \|f_{0}-\hf_{\lambda}\|^2+\|\hf_{\lambda}-\Pi_{S}\hf_{\lambda}\|^2+[\dim(S)\vee \Delta(S)]\sigma^2+\tilde\Sigma+\pa{\|\eps\|^2-2n\sigma^2}_{+} 
\end{equation}
for some positive constant $C$ depending on $K$ only. 
Finally, choosing $K=1.1$, we deduce the upper bound 
$$\E\cro{\tilde\Sigma+\pa{\|\eps\|^2-2n\sigma^2}_{+}}\leq 2\Sigma+3 \sigma^2,\quad \textrm{ (with $\Sigma$ defined in~(\ref{eq:sum}))}$$
 from the definition of $\pen_{\Delta}(S)$ and the fact that $\norm{Y-\Pi_{\overline S}Y}^{2}$ is independent of $\norm{\Pi_{\overline S}\eps}^{2}$ and is stochastically larger than $\norm{\eps-\Pi_{\overline S}\eps}^{2}$.
The bound~(\ref{eq:risque_esp}) follows.

\subsection{Proof of the second part of Proposition \ref{ORACLE-LINSELECT}}\label{proof:risque-proba}
We use the same notation as in Section~\ref{proof:risque-esp}. 
By (\ref{eq:corollaire_generale}), we have
\begin{eqnarray*}
C^{-1}\norm{f_{0}-\widehat f_{\widehat \lambda}}^{2} &\hspace{-0.2cm}\le&\hspace{-0.1cm}
\inf_{\lambda\in\Lambda}\ac{ \|f_{0}-\hf_{\lambda}\|^2+ \inf_{S\in\hS}\ac{\|\hf_{\lambda}-\Pi_{S}\hf_{\lambda}\|^2+[\dim(S)\vee \Delta(S)]\sigma^2}} \\
&& +\tilde\Sigma+\pa{\|\eps\|^2-2n\sigma^2}_{+} 
\end{eqnarray*}
for some positive constant $C$ depending on $K$ only. 
Setting $K=1.02$, we shall prove that with overwhelming probability $(\|\epsilon\|^2-2n\sigma^2)_{+}$ and 
\begin{equation*}
\tilde{\Sigma}:= 2\sum_{S\in\S}\left(1.02\|\Pi_{\overline{S}}\epsilon\|^2-\frac{\pen(S)}{n-\dim(S)}\|Y-\Pi_{\overline{S}}(Y)\|^2\right)_+\end{equation*}
are non positive. Applying a classical deviation inequality for $\chi^2$ random variables (Lemma 1 in \cite{Laurent00}), we derive that $\mathbb{P}\left[\|\epsilon\|^2\geq 2n\sigma^2\right]\leq e^{-n/16}$. Let us turn to $\tilde{\Sigma}$. The random variable ${(n-\dim(S)-1)\|\Pi_{\overline{S}}\epsilon\|^2}/{\|Y-\Pi_{\overline{S}}(Y)\|^2}$ is stochastically smaller than a variable $F_S$ such that $F_S/(\dim(S)+1)$ follows a Fisher distribution with $\dim(S)+1$ and $n-\dim(S)-1$ degrees of freedom. As a consequence, we have 
\begin{eqnarray}\label{eq_majoration_principale_borne_proba}
 \P\left[\tilde{\Sigma}>0 \right]\leq \sum_{S\in\mathbb{S}}\P\left[F_{S}\geq \frac{1.1}{1.02}\frac{n-\dim(S)-1}{n-\dim(S)}\pen_{\Delta}(S)\right].
\end{eqnarray}
 In order to upper bound the right hand-side of~(\ref{eq_majoration_principale_borne_proba}), we control the penalty terms $\pen_\Delta(S)$.
We have
\begin{equation*}
\E\cro{\pa{U-\frac{n-\dim(S)}{n-\dim(S)-1}\pen_{\Delta}(S)W}_{+}}=e^{-\Delta(S)}\ ,
\end{equation*}
where $U$ and $(n-\dim(S)-1)W$ are two independent $\chi^{2}$ random variables
with respectively $\dim(S)+1$ and $n-\dim(S)-1$ degrees of freedom. We prove in the next sections the three following technical lemmas.
\begin{lemma}\label{Lemma_minoration1_penalite}
Let $F=U/W$ and $0<\alpha<1$. We have
$$\P\pa{F\geq {1\over 1-\alpha} \frac{n-\dim(S)-1}{n-\dim(S)}\pen_{\Delta}(S)}\leq {e^{-\Delta(S)}\over \alpha (\dim(S)+1)}\ .$$
\end{lemma}
\medskip

\begin{lemma}\label{Lemma_minoration2_penalite}
Assume that $\dim(S)\leq n/2-1$. 
For any $u>1$ and for any $x\geq 0$, we have
\begin{eqnarray*}
 \P\pa{F\geq ux}\leq \exp\left[-\frac{u-1}{12u}\left\{(x-\dim(S)-1)\wedge n\right\}\right]\P\pa{F\geq x}\ .
\end{eqnarray*}
\end{lemma}
\medskip

\begin{lemma}\label{lemma_minoration_penalite}
For all $S\in\S$, we have
\begin{equation*}
\frac{n-\dim(S)-1}{n-\dim(S)} \pen_\Delta(S)\geq  2\Delta(S)+ \dim(S)-C\,,
\end{equation*}
where $C$ is a positive constant.
\end{lemma}
We can now complete the proof of Proposition \ref{ORACLE-LINSELECT}. 
Applying Lemma \ref{Lemma_minoration1_penalite} with $1/(1-\alpha)=1.1/1.05$ and Lemma \ref{Lemma_minoration2_penalite} with $u=1.05/1.02$ and 
\[x_S=\frac{1.1}{1.05}\times\frac{n-\dim(S)-1}{n-\dim(S)}\pen_{\Delta}(S)\ ,\] 
we derive from (\ref{eq_majoration_principale_borne_proba}) the following upper bound.
\begin{eqnarray*}
\P\left[\tilde{\Sigma}> 0\right]&\leq &\sum_{S\in\mathbb{S}} \exp\left[-C_2\left(\{x_S-\dim(S)-1\}\wedge n\right)\right]
\P\left[F_{S}\geq x_S\right]\\
&\leq &\sum_{S\in\mathbb{S}} C_1\exp\left[-C_2\left(\{x_S-\dim(S)-1\}\wedge n\right)\right]e^{-\Delta(S)}\\
&\leq &\sum_{S\in\mathbb{S}} C_1\exp\left[-C_2\left(\Delta(S)\wedge n\right)\right]e^{-\Delta(S)}.
\end{eqnarray*}
The proof of the second part of Proposition \ref{ORACLE-LINSELECT} is complete.

\subsection{Proof of the technical Lemmas  \ref{Lemma_minoration1_penalite}, \ref{Lemma_minoration2_penalite} and  \ref{lemma_minoration_penalite}}
\subsubsection{Proof of Lemma \ref{Lemma_minoration1_penalite}}\ \\
Since $U$ is independent of $W$ and $x\to (1-y/x)_{+}$ is increasing for all $y>0$ we have
\begin{eqnarray*}
e^{-\Delta(S)}&=& \E\cro{U\left(1-\frac{n-\dim(S)-1}{n-\dim(S)}\pen_{\Delta}(S)W/U\right)_{+}}\\
&\geq & \E[U]\,\E\left[\left(1-\frac{n-\dim(S)-1}{n-\dim(S)}\pen_{\Delta}(S)/F\right)_{+}\right]\\
&\geq& (\dim(S)+1)\times \alpha\, \P\left(1-\frac{n-\dim(S)-1}{n-\dim(S)}\pen_{\Delta}(S)/F\geq \alpha\right).
\end{eqnarray*}

\subsubsection{Proof of Lemma \ref{Lemma_minoration2_penalite}}\ \\
Note that the bound is trivial if $x\leq \dim(S)+1$. In the sequel, we assume that $x\geq \dim(S)+1$.
We set $d_1= \dim(S)+1$, $d_2=n-\dim(S)-1$ and write $B(.,.)$ for the Beta function. Since $d_1F$ follows a Fisher distribution with $(d_1,d_2)$ degrees of freedom, we have
%the density of $F$ is $t^{d_1/2}d_2^{d_2/2}(t+d_2)^{-(d_1+d_2)/2}t^{-1}B^{-1}(d_1/2,d_2/2)$ where . Hence,
\begin{eqnarray*}
 \P\pa{F\geq ux}&=&\int_{ux}^{+\infty}\frac{t^{d_1/2}d_2^{d_2/2}}{(t+d_2)^{(d_1+d_2)/2}tB(d_1/2,d_2/2)}dt\\  &= &\int_{x}^{+\infty}\frac{(ut)^{d_1/2}d_2^{d_2/2}}{(ut+d_2)^{(d_1+d_2)/2}tB(d_1/2,d_2/2)}dt\\
&\leq & u^{d_1/2} \int_{x}^{+\infty}\left[\frac{t+d_2}{ut+d_2}\right]^{(d_1+d_2)/2}\frac{t^{d_1/2}d_2^{d_2/2}}{(t+d_2)^{(d_1+d_2)/2}tB(d_1/2,d_2/2)}dt \\ & \leq& u^{d_1/2}\left[\frac{x+d_2}{ux+d_2}\right]^{(d_1+d_2)/2}\P\pa{F\geq x}\\
&\leq & \left\{u^{d_1/2}\left[\frac{d_1+d_2}{ud_1+d_2}\right]^{(d_1+d_2)/2}\right\}\left\{\left[\frac{(x+d_2)(ud_1+d_2)}{(ux+d_2)(d_1+d_2)}\right]^{(d_1+d_2)/2}\right\}\\ &\times &\P\pa{F\geq x}\ .
\end{eqnarray*}
In order to conclude, we shall prove that the first term between brackets is smaller than one and we shall control the second term.
The derivative of the function 
\[g: u\mapsto \log\left[u^{d_1/2}\left[\frac{d_1+d_2}{ud_1+d_2}\right]^{(d_1+d_2)/2}\right]\ \ \text{ is }\ \  g'(u)=\frac{d_1}{2}\left[\frac{1}{u}-\frac{d_1+d_2}{ud_1+d_2} \right]\ ,\]
which is non positive for any $u\geq 1$. Since $g(1)=0$, we conclude that the first term is smaller than one. Let us turn to the logarithm of the second term:
\begin{eqnarray*}
 -\frac{d_1+d_2}{2}\log\left[\frac{ux+d_2}{x+d_2}\frac{d_1+d_2}{ud_1+d_2}\right]&=& -\frac{d_1+d_2}{2}\log\left[1+\frac{d_2(u-1)(x-d_1)}{(x+d_2)(ud_1+d_2)}\right]\\
&\hspace{-1.5cm}\leq &\hspace{-1cm}-\frac{d_1+d_2}{2}\frac{d_2(u-1)(x-d_1)}{(x+d_2)(ud_1+d_2)+ d_2(u-1)(x-d_1)}\\
%&\leq &-\frac{d_1+d_2}{2u}\frac{d_2(u-1)(x-d_1)}{(x+d_2)(d_1+d_2)+d_2(x-d_1)} \\
&\hspace{-1.5cm}\leq &\hspace{-1cm} -\frac{(u-1)}{2u}(x-d_1) \left[\frac{x}{d_2}+1 +  \frac{x-d_1}{d_2+d_1}\right]^{-1}\\
&\hspace{-1.5cm}\leq &\hspace{-1cm} -\frac{(u-1)}{4u}\left[\frac{x-d_1}{2}\wedge \frac{n}{3}\right]\ ,
\end{eqnarray*}
where the last line is proved by considering separately $x\leq d_1+d_2$ and $x> d_1+d_2$ and by using $d_1\leq d_2\leq n/2$.

\subsubsection{Proof of Lemma \ref{lemma_minoration_penalite}}\ \\
We recall that the penalty $\pen_{\Delta}(S)$ is defined by
\begin{equation*}
\E\cro{\pa{U-{\pen_{\Delta}(S)\over n-\dim(S)}V}_{+}}=e^{-\Delta(S)}
\end{equation*}
where $x_{+}$ denotes the positive part of $x\in\mathbb{R}$ and $U,V$ are two independent $\chi^{2}$ random variables
with respectively $\dim(S)+1$ and $n-\dim(S)-1$ degrees of freedom. Let us lower bound this expectation applying Jensen's inequality.
\begin{eqnarray*}
\E\cro{\pa{U-{\pen_{\Delta}(S)\over n-\dim(S)}V}_{+}}&\geq &\E\cro{\pa{U-\frac{n-\dim(S)-1}{n-\dim(S)}\pen_{\Delta}(S)}_{+}}\\
& \geq & \mathbb{E}\left[\mathbf{1}\left\{{U>\frac{n-\dim(S)-1}{n-\dim(S)}\pen_{\Delta}(S)+1}\right\}\right]\ ,
\end{eqnarray*}
where $1\{A\}$ stands for the indicator function of the event $A$. Hence, we get
\begin{equation}\label{eq_preuve_lemma_minoration_penalite}
\pen_\Delta(S)\geq  \frac{n-\dim(S)}{n-\dim(S)-1}\left[\bar{\chi}^{-1}_{dim(S)+1}\left[e^{-\Delta(S)}\right]-1\right]\ ,
\end{equation}
where $\bar{\chi}^{-1}_{\dim(S)+1}(\alpha)$ is a $1-\alpha$ quantile of a $\chi^2$ random variable with $\dim(S)+1$ degrees of freedom.  \\

Let us note $k=\dim(S)+1$. For any positive number $x$, we have
\begin{eqnarray*}
\mathbb{P}\left[U\geq x+k\right]&=&\int_{x+k}^{+\infty}\frac{t^{k/2-1}e^{-t/2}}{2^{k/2}\Gamma(k/2)}dt= e^{-(x+k)/2}\int_{0}^{+\infty}\frac{(t+x+k)^{k/2-1}e^{-t/2}}{2^{k/2}\Gamma(k/2)}dt\\
&\geq & e^{-(x+k)/2}\frac{k^{k/2-1}}{2^{k/2}\Gamma(k/2)}\int_0^{\sqrt{k}} \exp\left[-\frac{t}{2}+\left(\frac{k}{2}-1\right)\log\left(1+\frac{t}{k}\right)\right]dt\\
&\geq & e^{-(x+k)/2}\frac{k^{k/2-1}}{2^{k/2}\Gamma(k/2)}\int_0^{\sqrt{k}}\exp\left[-\frac{t}{k}- \left(\frac{k}{2}-1\right)\frac{t^2}{2k^2}\right]dt\ ,
\end{eqnarray*}
since $\log(1+t)\geq t-t^{2}/2$. It follows that 
\[\mathbb{P}\left[U\geq x+k\right]\geq e^{-(x+k)/2}\frac{k^{k/2-1}}{2^{k/2}\Gamma(k/2)} \int_0^{\sqrt{k}}e^{-1}e^{-t/(4\sqrt{k})}dt\geq  C e^{-(x+k)/2}\frac{k^{k/2-1/2}}{2^{k/2}\Gamma(k/2)}\ .\]
By Stirling's expansion $\Gamma(k/2)\leq (k/2)^{k/2-1/2}e^{-k/2}\sqrt{2\pi}$ so that $\mathbb{P}\left[U\geq x+k\right]\geq Ce^{-x/2}$.
It follows that 
\[\bar{\chi}^{-1}_{\dim(S)+1}\left(e^{-\Delta(S)}\right)\geq 2\Delta(S)+ \dim(S)+1-C\,.\]

\section{Proof of the specific bounds for Lasso-\LinSelect\ and Group-Lasso-\LinSelect}\label{proof:oracle-lasso-glasso}

\subsection{Size of the support of the Lasso and Group-Lasso estimators}\ \\
For $\K\subset \ac{1,\ldots,M}$, we recall that
$\phi_{(\K)}$ denotes the largest eigenvalue of $\X_{(\K)}^T\X_{(\K)}$. 
\begin{lemma}\label{LEMMETECH2}
Let $\hK_{\lambda}$ be the subset of groups selected by the group-Lasso estimator  $\widehat{\beta}_{\lambda}$. Then, on the event 
$\mathcal{A}_{\lambda}=\bigcap_{k=1}^M\ac{\|\X_{G_{k}}^T\eps\|_{2}\leq \lambda_{k}/4}$ we have
$$\sum_{k\in\hK_{\lambda}}\lambda^2_{k} \leq 16\, \phi_{(\hK_{\lambda})}\, \|\X\hbeta_{\lambda}-\X\beta_{0}\|^2_{2}\,.$$
In particular, for the Lasso estimator  $\hbeta_{\lambda}^L$, we have  the upper bound
$$\lambda^2 \|\hbeta_{\lambda}^L\|_{0} \leq 16\, \phi_{\mathrm{supp}(\hbeta_{\lambda}^L)}\, \|\X\hbeta_{\lambda}^L-\X\beta_{0}\|^2_{2}$$
on the event
$\mathcal{A}_{\lambda}= \ac{|\X^T\eps|_{\ell^\infty}\leq \lambda/4}$.
\end{lemma}
The proof of this lemma is delayed to the Appendix~\ref{proof:LEMMETECH2}.
The above bounds are similar to those stated in Bickel {\it et al.}~\cite{bickeltsy08} and Lounici {\it et al.}~\cite{tsybakovgroup}, except that it involves the restricted eigenvalue $\phi_{(\hK_{\lambda})}$ instead of the largest eigenvalue $\phi_{\max}$ of $X^TX$. When $|\hK_{\lambda}|$ is small compared to $n$ the restricted eigenvalue $\phi_{(\hK_{\lambda})}$ can be much smaller than $\phi_{\max}$. Actually, since $X^TX$ has at most $n$ non-zero eigenvalues and $\mathrm{Tr}(X^TX)=p$, we always have $\phi_{\max}\geq p/n$ which can be large when $p\gg n$.
% \textcolor{red}{commenter similaire \`a...}

\subsection{Proof of Proposition~\ref{prte_risque_proba_LinSelect_Lasso}}\ \\
The first step is to provide a sufficient condition for having $\|\hbeta_{\lambda}\|_{0}\leq n/(3\log(p))$.
Recall that the compatibility constant $\kappa[\xi,T]$ is defined in Section~\ref{result-lasso}.
\begin{lemma}\label{lem:cardinal-lasso}
Assume that $\lambda\geq 8\sigma\sqrt{\log(p)}$ and
\begin{equation}
 1\leq\|\beta_{0}\|_{0}\leq {\kappa^2[5,\supp(\beta_{0})] \over 96\,\phi_{*}}\times {n\over \log(p)}\ .\label{eq:hypothese_sparsite}
\end{equation}
Then, on the event $\mathcal{A}= \ac{|\X^T\eps|_{\ell^\infty}\leq 2\sigma\sqrt{\log(p)}}$ we have 
$\|\hbeta_{\lambda}\|_{0}\leq n/(3\log(p))$.
\end{lemma}
\begin{proof}[Proof of Lemma~\ref{lem:cardinal-lasso}]
We write $\hJ$ for the support of $\hbeta_{\lambda}$.
A slight variation of Theorem~14 in \cite{Kolt11} ensures that
\begin{equation}\label{KLT}
\|\X\hbeta_{\lambda}-\X\beta_{0}\|_2^2\leq \inf_{\beta\neq 0}\ac{\|\X\beta_{0}-\X\beta\|^2_{2}+{\lambda^2\over \kappa^2[5,\supp(\beta)]}\|\beta\|_{0}}
\end{equation}
on the event $\mathcal{A}$.
Combining Lemma~\ref{LEMMETECH2} with the bound (\ref{KLT}) we obtain that
$${\rm Card}(\hJ)\leq 16 \,\phi_{\hJ}\,{\|\beta_{0}\|_{0} \over \kappa^2[5,\supp(\beta_{0})]}.$$
Let us set $d^*=n/[3\log(p)]$.
The upper-bound
$\phi_{\hJ}\leq (1+ {\rm Card}(\hJ)/d^*)\phi_{*}$
enforces
$${\rm Card}(\hJ) \leq {16 \phi_{*} \|\beta_{0}\|_{0} \over \kappa^2[5,\supp(\beta_{0})]}\left[1+\frac{\mathrm{Card}(\hJ)}{d^*}\right]\leq\pa{d^*+\mathrm{Card}(\hJ)}/2\,,$$
where the last inequality follows from (\ref{eq:hypothese_sparsite}).
\end{proof}

We can now complete the proof of Proposition~\ref{prte_risque_proba_LinSelect_Lasso}. We recall that the event  $\mathcal{A}= \ac{|\X^T\eps|_{\ell^\infty}\leq 2\sigma\sqrt{\log(p)}}$ has probability at least $1-1/p$. 
Let us set 
$$\lambda_0=\sqrt{16(4\vee \phi_{*})\log(p) \sigma^2}\geq 8\sigma\sqrt{\log(p)}.$$
Under the hypothesis~(\ref{eq:hypothese_sparsite}), the combination of Lemma~\ref{lem:cardinal-lasso} with Proposition~\ref{ORACLE-LINSELECT} ensures that with probability larger than  
$1-C_1p^{-C_2}$ we have
\begin{eqnarray*}
\norm{{\bf X}\beta_{0}-{\bf X}\widehat \beta_{\widehat \lambda}}^{2}_{2} 
%&\le& C\inf_{\lambda\in\Lambda, S\in \hS}\left\{\|{\bf X}\beta_{0}-{\bf X}\widehat{\beta}_{\lambda}\|^2+\|{\bf X}\widehat{\beta}_{\lambda}-\Pi_{S}{\bf X}\widehat{\beta}_{\lambda}\|^2+[\Delta(S)\vee \dim(S)]\sigma^2\right\} \\
&\leq & C\big\{\|{\bf X}\beta_{0}-{\bf X}\widehat{\beta}_{\lambda_{0}}\|^2_{2}+[\|\widehat{\beta}_{\lambda_{0}}\|_0\vee 1]\log(p)\sigma^2\big\}.
\end{eqnarray*}
We upper bound the right-hand side by combining Lemma~\ref{LEMMETECH2} with~(\ref{KLT})
\begin{eqnarray*}
\norm{{\bf X}\beta_{0}-{\bf X}\widehat \beta_{\widehat \lambda}}^{2}_{2} &\leq&C \pa{1+{16\phi_{\hJ}\log(p)\sigma^2\over \lambda^2_{0}}} \\
&&\times\inf_{\beta\neq 0}\ac{\|\X\beta_{0}-X\beta\|^2_{2}+{\lambda^2_{0}\over \kappa^2[5,\supp(\beta)]}\|\beta\|_{0}}\\
&\leq& C'  \inf_{\beta\neq 0}\ac{\|\X\beta_{0}-X\beta\|^2_{2}+{\phi_{*}\log(p)\sigma^2\over \kappa^2[5,\supp(\beta)]}\|\beta\|_{0}},
\end{eqnarray*}
where we used in the last inequality that $\hJ$ (the support of $\hbeta_{\lambda_{0}}$) is of size at most $n/(3\log(p))$.

\subsection{Proof of Proposition~\ref{prop_risque_glasso}}\ \\
The proof of Proposition~\ref{prop_risque_glasso} is very similar to that of Proposition~\ref{prte_risque_proba_LinSelect_Lasso}. We only sketch the main lines.
The first step is to provide a sufficient condition for having $|\hK_{\lambda}|\leq (n-2)/(2T\vee 3\log(M))$.
Recall that the compatibility constant $\kappa_{G}[\xi,s]$ is defined in~(\ref{eq:kappa-glasso}) and $\phi_{*}$ in~(\ref{eq:phi-glasso}).
\begin{lemma}\label{lem:cardinal-glasso}
Assume that 
\begin{eqnarray}
\lambda_{k}^2&=&96\phi_{*}(T\vee 3\log(M))\sigma^2,\quad\mathrm{for}\ k=1,\ldots,M\label{lambda-glasso}\\
\mathrm{and}\quad 1&\leq& |\mathcal{K}_{0}|\ \leq\ \frac{\kappa^2_G[3,|\mathcal{K}_0|]}{2^{9}\phi_*}\times \frac{n-2}{2T\vee 3\log(M)}.\label{condition_group_lasso}
\end{eqnarray}
Then we have $|\hK_{\lambda}|\leq (n-2)/(3\log(M)\vee 2T)$, with probability at least $1-3/M$.
\end{lemma}
\begin{proof}[Proof of Lemma~\ref{lem:cardinal-glasso}]
We set $k^*=(n-2)/(3\log(M)\vee 2T)$.  Theorem 3.1 in \cite{tsybakovgroup} gives
\begin{equation}\label{eq:glasso-tsyb}
 \|{\bf X}\widehat{\beta}_{\lambda}-{\bf X}\beta_0\|_2^2\leq \frac{16}{\kappa^2_G[3,|\mathcal{K}_0|]}\,|\mathcal{K}_{0}|\, \lambda_{1}^2\, ,
\end{equation}
with probability larger than $1-3/M$.
Combining this bound with Lemma~\ref{LEMMETECH2} and the bound $\phi_{(\widehat{\mathcal{K}}_\lambda)}\leq [1+|\widehat{\mathcal{K}}_\lambda|/k_*]\phi_*$, we get that with probability larger than $1-3/M$  
\begin{eqnarray*}
|\widehat{\mathcal{K}}_{\lambda}|&\leq& \frac{2^{8}}{\kappa^2_G[3,|\mathcal{K}_0|]}\phi_{(\widehat{\mathcal{K}}_\lambda)}|\mathcal{K}_{0}|\\ 
&\leq&  \frac{2^{8}}{\kappa^2_G[3,|\mathcal{K}_0|]}\left[1+ \frac{|\widehat{\mathcal{K}}_\lambda|}{k_*}\right]\phi_*|\mathcal{K}_{0}|\ \leq \ (k^*+|\hK_{\lambda}|)/2\,,
\end{eqnarray*}
where the last bound follows from (\ref{condition_group_lasso}).
\end{proof}
We complete now the proof of Proposition~\ref{prop_risque_glasso}.
Assume that (\ref{lambda-glasso}) and (\ref{condition_group_lasso}) are satisfied.
Combining Lemma~\ref{lem:cardinal-glasso} with Proposition~\ref{ORACLE-LINSELECT} ensures that with probability larger than  
$1-C_1M^{-C_2}-3/M$ we have
\begin{eqnarray*}
C^{-1}\norm{{\bf X}\beta_{0}-{\bf X}\widehat \beta_{\widehat \lambda}}^{2}_{2} &\leq& 
\|{\bf X}\hbeta_{\lambda}-{\bf X}\beta_{0}\|_2^2+(1\vee  |\hK_{\lambda}|)\pa{T\vee \log(M)}\sigma^2\\
&\leq& 2  \pa{\|\X\hbeta_{\lambda}-\X\beta_{0}\|^2_{2}\vee\cro{\pa{T\vee \log(M)}\sigma^2}}\,.
\end{eqnarray*}
Proposition~\ref{prop_risque_glasso} then simply follows from (\ref{eq:glasso-tsyb}).

\subsection{Proof of Lemma~\ref{LEMMETECH2}}\label{proof:LEMMETECH2}
We write $\hbeta$ for $\hbeta_{\lambda}$, $\hK$ for $\hK_{\lambda}$ and $A^+$ for the Moore-Penrose pseudo-inverse of $A$. The optimality condition gives
\begin{equation}\label{KKT2}
2\X_{(\hK)}^T(Y-\X_{(\hK)}\hbeta_{(\hK)})=\lambda z_{(\hK)}
\end{equation}
where $\|z^{G_{k}}\|_{2}=1$ for all $k\in\hK$.
As a consequence we have
$$\hbeta_{(\hK)}=(\X_{(\hK)}^T\X_{(\hK)})^+(\X_{(\hK)}^TY-\lambda z_{(\hK)}/2)$$
and
$$\X\hbeta=P_{(\hK)}Y-\lambda \X_{(\hK)}(\X_{(\hK)}^T\X_{(\hK)})^+ z_{(\hK)}/2$$
where $P_{(\hK)}$ is the orthogonal projector onto the range of $\X_{(\hK)}$.
Pythagorean equality gives
\begin{eqnarray*}
\|\X\beta_{0}-\X\hbeta\|^2_{2}&=&\|\X\beta_{0}-P_{(\hK)}\X\beta_{0}\|^2_{2}+\|P_{(\hK)}\eps-\lambda \X_{(\hK)}(\X_{(\hK)}^T\X_{(\hK)})^+ z_{(\hK)}/2\|^2_{2}\\
&\geq & \|\X_{(\hK)}(\X_{(\hK)}^T\X_{(\hK)})^+ (\X_{(\hK)}^T\eps-\lambda z_{(\hK)}/2)\|^2_{2}.
\end{eqnarray*}
From (\ref{KKT2}) we know that the vector $\X_{(\hK)}^T\eps-\lambda z_{(\hK)}/2$ belongs to the range of 
$\X_{(\hK)}^T$ and therefore (see Lemma~\ref{tech1} below)
$$\phi_{(\hK)}\, \|\X_{(\hK)}(\X_{(\hK)}^T\X_{(\hK)})^+ (\X_{(\hK)}^T\eps-\lambda z_{(\hK)}/2)\|^2_{2}\geq  \|\X_{(\hK)}^T\eps-\lambda z_{(\hK)}/2\|^2_{2}.$$
Finally, on the event $\mathcal{A}_{\lambda}$ we have $\|\X_{G_{k}}^T\eps-\lambda_{k} z^{G_{k}}/2\|_{2}\geq \lambda_{k}/4$ for all $k\in\hK$, so 
$$\|\X_{(\hK)}^T\eps-\lambda z_{(\hK)}/2\|^2_{2}\geq \sum_{k\in\hK}\lambda^2_{k} /16.$$
This allows to conclude.
\begin{lemma}\label{tech1}
Let $A$ be any $n\times d$ real matrix. Then for any $x$ in the range of $A^T$ we have
$$\|x\|^2_{2}\leq \varphi_{\max}(A^TA) \,\|A(A^TA)^+x\|^2_{2}$$
where $\varphi_{\max}(A^TA)$ denotes the largest eigenvalue of $A^TA$.
\end{lemma}
\begin{proof}[Proof of Lemma~\ref{tech1}]
We first note that
$$\|A(A^TA)^+x\|^2_{2}=x^T(A^TA)^+A^TA(A^TA)^+x=x^T(A^TA)^+x.$$
%The vector $x$ is of the form $x=A^Ty$.
% Since $\R^n={\rm range}(A)\oplus{\rm ker}(A^T)$ we have $y=Az+r$ where $r\in{\rm ker}(A^T)$. Therefore 
% $x=A^TAz$ lies in the range of $A^TA$. 
Furthermore the range of $A^T$ coincides with the range of $A^TA$, which in turn is the same as the range of $(A^TA)^+$. %(these facts can be easily checked  via a singular value decomposition). 
We then have
$$\sigma_{{\rm rank}((A^TA)^+)}((A^TA)^+) \|x\|^2_{2}\leq x^T(A^TA)^+x$$
where $\sigma_{k}((A^TA)^+)$ is the $k$-th largest singular value of $(A^TA)^+$. The result follows from the equality
$$\left[\sigma_{{\rm rank}((A^TA)^+)}((A^TA)^+)\right]^{-1}=\sigma_{1}(A^TA)=\varphi_{\max}(A^TA).$$
\end{proof}

\end{document}